\documentclass[11pt,reqno]{amsart}

\usepackage[T1]{fontenc}
\usepackage[utf8]{inputenc}
\usepackage{amsmath,amssymb,amsthm,mathrsfs,bm,empheq}
\usepackage{booktabs,graphicx}
\usepackage{xcolor}
\usepackage{enumitem}
\usepackage{geometry}
\usepackage{hyperref}
\hypersetup{hidelinks}
\allowdisplaybreaks[4]

\theoremstyle{plain}
\newtheorem{lemma}{Lemma}[section]
\newtheorem{theorem}[lemma]{Theorem}
\newtheorem{proposition}[lemma]{Proposition}

\newtheorem{assumption}{Assumption}[section]
\newtheorem{example}[lemma]{Example}
\theoremstyle{remark}
\newtheorem{remark}[lemma]{Remark}

\newcommand{\cP}{\mathcal P}
\newcommand{\cH}{\mathcal H}
\newcommand{\cL}{\mathcal L}
\newcommand{\cF}{\mathcal F}
\newcommand{\cU}{\mathcal U}
\newcommand{\bbR}{\mathbb R}
\newcommand{\bbS}{\mathbb S}
\newcommand{\gradG}{\nabla_G}
\newcommand{\divG}{\operatorname{div}_G}
\newcommand{\DeltaG}{\Delta_G}
\newcommand{\ip}[2]{(#1,#2)}
\newcommand{\norm}[1]{\lVert#1\rVert}
\makeatletter
\@namedef{subjclassname@2020}{%
  \textup{2020} Mathematics Subject Classification}
\makeatother

\begin{document}

\title[A staggered scheme for graph MFG]{Convergence and variational structure of a staggered scheme for mean field games with individual noise on graphs
} 
\author{Jianbo Cui, Tonghe Dang}
\address{Department of Applied Mathematics, The Hong Kong Polytechnic
University, Hung Hom, Kowloon, Hong Kong, SAR, China}
\email{jianbo.cui@polyu.edu.hk; tonghe.dang@polyu.edu.hk(Corresponding author)}
\thanks{This work is supported by MOST National Key R\&D Program No. 2024FA1015900, the Hong Kong Research Grant Council GRF grant 15302823, GRF grant 15301025, NSFC/RGC Joint Research Scheme N$\_$PolyU5141/24, NSFC grant 12522119, NSFC grant 12301526, internal funds (P0041274, P0045336)  from Hong Kong Polytechnic University,  and the
CAS AMSS-PolyU Joint Laboratory of Applied Mathematics.} 

\begin{abstract}
We propose and analyze a time-staggered numerical scheme for mean field game (MFG) systems with individual noise on finite graphs.  Numerically solving such coupled forward--backward systems is delicate because the density evolves in the open probability simplex and the
coefficients may degenerate at its boundary.   The scheme
preserves mass and satisfies a discrete fundamental identity compatible with the Lasry--Lions monotonicity argument, leading to uniqueness of the numerical solution.  By establishing a timestep-uniform positive lower
bound for the density and uniform bounds for the value variable, we 
prove first-order convergence for every interior discrete solution. 
For potential MFGs, we establish a variational characterization by
identifying the scheme with the KKT system of a convex discrete action, 
yielding existence of the discrete solution and an optimization-based
realization.  The resulting optimization problem is solved by a feasible primal--dual Newton method in mass-preserving coordinates. Numerical experiments
confirm the predicted convergence rate and illustrate
topology-dependent transport and congestion-driven route
choice. 

\end{abstract}

\maketitle

\section{Introduction}

Mean field games (MFGs) describe Nash equilibria of differential games with a
large population of weakly interacting agents
\cite{LasryLions2007,HuangMalhameCaines2006}. A representative agent optimizes
its own objective against a prescribed population distribution, while the
equilibrium condition requires that the induced population distribution agree
with the prescribed one. This leads to a coupled forward--backward system
consisting of a Hamilton--Jacobi equation for the value function and a
Fokker--Planck equation for the population density
\cite{CardaliaguetPorretta2020,CarmonaDelarue2018I,
BensoussanFrehseYam2013}. Finite-state MFGs arise
when the admissible states are discrete and the transitions between them are
determined by a graph. Such models naturally describe decision processes on
transportation networks, communication networks, regime-switching systems, and
other problems with discrete state spaces
\cite{GomesMohrSouza2010,GomesMohrSouza2013,Gueant2015,CecchinFischer2020,CarmonaWang2021,BonnansLavignePfeiffer2023}.

In this paper, we consider the following MFG system \cite{GangboMunozWuZhang2026} with individual noise posed on a finite
connected weighted graph $G=(V,E,\omega)$: 
\begin{align}
\begin{cases}\label{mfg}
\dot\phi=\bm H(\rho,\gradG\phi)-\DeltaG\phi,\\
\dot\rho=\divG\bm B(\rho,\gradG\phi)+\DeltaG\rho,\\
\rho(0)=\rho_0,\qquad
\phi(T)=g(\rho(T)).
\end{cases}
\end{align}
Here \(V\), \(E\), and \(\omega\) denote the vertex set, edge set, and edge weights, respectively. The variable $\rho$ is a probability vector on the graph vertices, $\phi$ is the
value variable, $\nabla_G$ is the graph gradient operator, and the graph Laplacian $\Delta_G$ models the independent Markov jumps
of individual agents. The coefficients $\bm H$ and $\bm B$ denote,
respectively, the Hamiltonian and the transport flux; see Section \ref{sec2} for precise definitions. 
For potential MFGs, \(\bm H\), \(\bm B\), and \(g\) are generated by a scalar Hamiltonian, a running potential, and a terminal potential.
The model \eqref{mfg} describes MFGs driven simultaneously by controlled transport
and independent Markov jumps of individual agents, and can be viewed as
the finite-state counterpart of second-order MFGs with diffusion on
Euclidean domains.

Unlike Euclidean MFGs, where agents move in a continuous state space, the graph itself is the state space of the dynamics.  It therefore induces a finite-dimensional geometric structure on the probability simplex; see Section~\ref{sec2-compare} for a comparison of the two settings.
  In particular, 
the logarithmic mean
endows the probability simplex of a finite Markov chain with a
Wasserstein-type geometry, under which the graph heat equation becomes the
gradient flow of the entropy
\cite{ChowHuangLiZhou2012,Maas2011,ChowLiZhou2018,ChowDieciLiZhou2019}. This geometric framework has led to the
development of graph Wasserstein geodesics
\cite{GangboLiMou2019,ErbarMaas2012,Mielke2011,Mielke2013,CuiDieciZhouSISC2022,ErbarRumpfSchmitzerSimon2020}, Hamiltonian flows
\cite{ChowLiZhou2019,ChowLiZhou2020,CuiDieciZhou2022,CuiLiuZhou2023}, and
Hamilton--Jacobi equations on the probability simplex
\cite{GangboMouSwiech2024,CuiDangMou2026,CuiDang2025}. More recently, the well-posedness of the
graph MFG \eqref{mfg}, together with the associated master equation
and mean field control problem, has been established
\cite{GangboMunozWuZhang2026}. These developments naturally raise the
question of developing numerical
methods that preserve the intrinsic mathematical structures of graph
MFGs, including the monotonicity mechanism and probability-simplex structure.

The numerical analysis of MFGs on Euclidean domains has been extensively
developed over the past decades. For instance, finite-difference discretizations for second-order MFGs were
introduced in \cite{AchdouCapuzzoDolcetta2010}, and their convergence was
established in \cite{AchdouCamilliCapuzzoDolcetta2013}. For planning problems, variational formulations together with Newton-type solvers were developed in \cite{AchdouCamilliCapuzzoDolcetta2012}. 
Further optimization and computational methods
include proximal, generalized conditional-gradient, and primal-dual approaches
for continuous-state MFGs 
\cite{BricenoAriasKaliseSilva2018,LavignePfeiffer2023,
LiuJacobsLiNurbekyanOsher2021}. 
Recent developments  also  include finite element methods for
time-dependent MFGs with nondifferentiable Hamiltonians \cite{OsborneSmears2025}
and particle methods for first-order MFGs under displacement monotonicity
\cite{MeszarosOsborne2026}. 
In contrast, numerical analysis for finite-state graph MFGs is substantially less developed than for Euclidean MFGs. For finite-state MFGs satisfying a monotonicity condition, a contractive artificial-time flow was constructed in \cite{GomesSaude2021}. More recently, an initial-value optimization formulation together with a neural-network implementation was proposed in \cite{FengXiangZhou2026}. To the best of our knowledge, however, no time discretization of the individual-noise graph MFG \eqref{mfg} has been analyzed. 
 In
particular, neither a timestep-uniform interior estimate nor a
first-order convergence result is currently available.

The difficulty in the numerical analysis of \eqref{mfg} stems from three coupled
features of the system.
First, the value equation evolves backward in time, whereas the density
equation evolves forward in time, and the two equations are nonlinearly
coupled. Consequently, neither variable can be computed independently by
a standard time-marching procedure. 
Second, the density evolves in the open probability simplex. For the
logarithmic-mean model, the mobility degenerates as the density
approaches the boundary of the simplex, and the nonlinear coefficients
lose their uniform regularity. Positivity at each time level
is therefore insufficient. Instead, a timestep-independent interior
estimate is needed to control the nonlinear coefficients uniformly and
to establish convergence. Third, the time levels in the two equations must be chosen consistently 
so that the discretization preserves the intrinsic analytical
structure of the continuous problem. In particular, the discrete scheme
should retain the forward--backward monotonicity mechanism underlying
the Lasry--Lions argument and, in the potential case, admit a 
variational characterization.
Reconciling these
requirements is the key to obtaining a stable, convergent, and
structure-preserving numerical method.

Motivated by these observations, we introduce a time-staggered
discretization in which the nonlinear terms in both equations are evaluated at $(\rho^{n+1},\nabla_G\phi^n)$. The placement is chosen so that the discrete forward--backward identity retains the Lasry--Lions monotonicity mechanism, which provides the key
ingredient for proving uniqueness of the numerical solution. 
For general monotone coefficients, we combine the discrete fundamental
identity with a graph maximum principle to derive timestep-uniform
interior estimates for the density together with uniform bounds for the
value variable and its graph gradient. These estimates yield first-order convergence for every interior discrete solution. 
For potential graph MFGs, we further construct a convex discrete action
whose KKT system is shown to be equivalent to the proposed scheme. 
This variational characterization additionally yields the existence of an interior discrete solution and 
provides an optimization-based numerical realization.  Finally, we develop a feasible primal--dual Newton method for solving the resulting optimization problem. Each accepted iterate preserves the discrete
continuity equation and keeps the density in the open probability
simplex. 
Numerical experiments exhibit first-order temporal convergence consistent with the predicted rate and illustrate topology-dependent transport and congestion-driven route selection.

The paper is organized as follows.  Section~\ref{sec2} introduces the graph operators,
the continuous MFG system, the logarithmic-mean quadratic model, and the comparison with
Euclidean MFGs.
Section~\ref{sec3} presents the time-staggered scheme and the main result, including the uniform regularity estimate, convergence order, and variational formulation. 
Section~\ref{sec4} proves uniqueness, uniform
interior estimates, and first-order convergence for the general MFG scheme.
Section~\ref{sec5} establishes the variational characterization and existence of the
discrete solution in the potential case.   Section~\ref{sec6} gives the feasible
Newton realization and the numerical experiments.  Additional proofs are
collected in the appendix.

\section{Preliminaries and MFGs on Graphs}\label{sec2}

This section introduces the notation and the continuous MFG system.
We also state the assumptions used in the analysis and then verify them for the
logarithmic-mean quadratic model in Example~\ref{exam}.   Finally, we compare the graph setting with MFGs on Euclidean domains. 

\subsection{Notation}
Let $G=(V,E,\omega)$ be a finite connected undirected graph without self-loops,
where $V=\{1,\ldots,d\}$ is the vertex set with $d\ge2$, $E\subset V\times V$ is the edge set,    and $\omega=(\omega_{ij})_{1\leq i,j\leq d}$ is a symmetric weight matrix satisfying $\omega_{ii}=0$, $\omega_{ij}=\omega_{ji}>0$ for $(i,j)\in E$, and $\omega_{ij}=0$ for $(i,j)\notin E$.   
Throughout the paper, the set $E$ contains both orientations of
each undirected edge.  We write $j\sim i$ if
$(i,j)\in E$, set
$\omega_{\min}=\min_{(i,j)\in E}\omega_{ij}$ and
$\omega_{\max}=\max_{(i,j)\in E}\omega_{ij}$, and denote by
$\bm1=(1,\ldots,1)^T\in\bbR^d$ the constant vector. 
For $x,y\in\bbR^d$, let
$\ip{x}{y}=\sum_{i=1}^d x_i y_i$,
$\norm{x}=\ip{x}{x}^{1/2}$, and
$\norm{x}_{\ell^\infty}=\max_{i=1,\ldots,d}|x_i|$. 
Let $\bbS^{d\times d}$ be the space of skew-symmetric matrices $p=(p_{i,j})_{1\leq i,j\leq d}$ satisfying
$p_{ji}=-p_{ij}$ and $p_{ij}=0$ when $(i,j)\notin E$.  It is equipped with
\begin{align}\label{infty-E}
 \ip{p}{q}_E=\frac12\sum_{(i,j)\in E}p_{ij}q_{ij},
 \qquad
 \norm{p}_E=\ip{p}{p}_E^{1/2},
 \qquad
 \norm{p}_{\ell^\infty(E)}=\max_{(i,j)\in E}|p_{ij}|, 
\end{align}
where $q\in\mathbb S^{d\times d}.$ 
For $u\in\bbR^d$, define the graph gradient
$(\gradG u)_{ij}:=\sqrt{\omega_{ij}}(u_i-u_j)$ for $(i,j)\in E$.
For $q\in\mathbb S^{d\times d},$ the graph divergence is defined by \begin{equation}\label{graph-divergence}
 (\divG q)_i=-\sum_{j\sim i}\sqrt{\omega_{ij}}q_{ij},
\end{equation} and satisfies the discrete integration-by-parts identity 
\begin{equation}\label{graph-ibp}
 \ip{\divG q}{u}=-\ip{q}{\gradG u}_E.
\end{equation} 
The graph Laplacian is $\DeltaG=\divG\gradG$ and thus 
\begin{equation}\label{graph-laplacian}
 (\DeltaG u)_i=\sum_{j\sim i}\omega_{ij}(u_j-u_i),
 \qquad
 \ip{\DeltaG u}{v}=-\ip{\gradG u}{\gradG v}_E.
\end{equation}
The probability simplex is
$\cP(G)=\{\xi\in[0,1]^d:\sum_{i=1}^d\xi_i=1\},$ and its
interior is $\cP^\circ(G)=\{\xi\in\cP(G):\xi_i>0\}$.
For $\epsilon\in(0,1/d)$, write
$\cP_\epsilon=\{\xi\in\cP(G):\xi_i\ge\epsilon\}$.
We use $\xi$ for a generic point of the simplex and $\rho$ for a density trajectory.  
The tangent space of the simplex is $\bbR_0^d:=\{\eta\in\bbR^d:\sum_{i=1}^d\eta_i=0\}$.
The symbols $D_\rho,D_p$ denote Fr\'echet derivatives in the
indicated variables, and   $D^2$ denotes the corresponding Hessian. For scalar functions defined on the simplex, density derivatives are
intrinsically covectors on the tangent space $\mathbb R_0^d$.
Whenever derivatives such as
$D_\rho\mathcal H$,
$D_\rho\mathcal L$,
$D_\rho\mathcal F$,
or
$D_\rho\mathcal U_T$
are represented by vectors in $\mathbb R^d$, we fix smooth extensions
to a neighborhood of the simplex. Different choices of extension differ only by an additive multiple of
$\mathbf1$, and therefore define the same linear functional on
$\mathbb R_0^d$.

\subsection{MFGs on graphs} 
Let $\bm H:\cP^\circ(G)\times\bbS^{d\times d}\to\bbR^d$ and
$\bm B:\cP^\circ(G)\times\bbS^{d\times d}\to\bbS^{d\times d}$ be the
Hamiltonian and transport coefficients, respectively, and let
$g:\cP(G)\to\bbR^d$ be a terminal functional.  Fix $T>0$.  Throughout this paper, we consider the graph MFG system
\eqref{mfg}, where $\rho_0\in\cP^\circ(G)$, 
and the coefficients satisfy the following regularity and coercivity conditions (see
\cite[(3.1)--(3.4)]{GangboMunozWuZhang2026}).
\begin{assumption}
\label{continuous-data}
Let $\bm H\in\mathcal C^1(\cP^\circ(G)\times\bbS^{d\times d};\bbR^d)$, $\bm B\in\mathcal C^1(\cP^\circ(G)\times\bbS^{d\times d};\bbS^{d\times d})$,   
and $g\in\mathcal C^1(\cP(G);\bbR^d)$.  There are constants $C_0,C_H,C_g\ge0$ such that
\begin{equation}\label{data-bounds}
 \ip{\bm B(\rho,p)}{p}_E
 \ge\ip{\bm H(\rho,p)}{\rho}-C_0,
 \qquad
 \bm H_i(\rho,p)\ge-C_H,
 \qquad
 \norm{g(\rho)}_{\ell^\infty}\le C_g.
\end{equation}
Moreover, there is a locally bounded function
$\mathfrak a:(0,\infty)\times\bbS^{d\times d}\to[0,\infty)$ such that
$|\bm B_{ij}(\rho,p)|\le(\rho_i+\rho_j)
\mathfrak a(\rho_j/\rho_i,p)$ for $(i,j)\in E$,
and $\mathfrak a(u,p)\to0$ as $u\to0^+$ or $u\to\infty$, locally uniformly in
$p\in\bbS^{d\times d}$.
\end{assumption} 

For $R>0$, set 
$h_R(u):=\sup_{\norm{p}_{\ell^\infty(E)}\le R}\mathfrak a(u,p)$, 
where $\|\cdot\|_{\ell^{\infty}(E)}$ is given in \eqref{infty-E}.  
Assumption~\ref{continuous-data} implies that $h_R$ is bounded, tends to zero
as $u\to0^+$ or $u\to\infty$, and satisfies
\begin{equation}\label{bounded-mobility}
 |\bm B_{ij}(\rho,p)|
 \le(\rho_i+\rho_j)h_R(\rho_j/\rho_i)
 \quad\text{whenever }\norm{p}_{\ell^\infty(E)}\le R.
\end{equation} 
Indeed, according to Assumption~\ref{continuous-data}, the convergence
of $\mathfrak a(u,p)$ to zero as $u\to0^+$ or $u\to\infty$ is uniform
for $\norm{p}_{\ell^\infty(E)}\le R$. Hence there exist
$u_->0$ and $u_+>u_-$ such that $ h_R(u)\le1$ 
whenever $ 0<u\le u_-
\text{ or }
u\ge u_+.$  On the remaining compact interval $[u_-,u_+]$, local boundedness of $\mathfrak a$
gives a uniform bound on
$[u_-,u_+]\times
\{p:\norm{p}_{\ell^\infty(E)}\le R\}.$  Therefore, $ 
\sup_{u>0}h_R(u)<\infty.$  
Moreover, \eqref{bounded-mobility} and property of $h_R$ implies that for every $(i,j)\in E$, the component $\bm B_{ij}$ can be continuously
extended to boundary points satisfying $\rho_i\rho_j=0$ by setting \begin{equation}\label{boundary-B_revise}
 \bm B_{ij}(\rho,p)=0
 \quad\text{whenever }\;\rho_i\rho_j=0.
\end{equation}
This will be used in the positivity estimate for the numerical
density. 

Assumption \ref{continuous-data} ensures the existence of a solution to the MFG \eqref{mfg}. For uniqueness, \cite{GangboMunozWuZhang2026} introduces the terminal and Lasry--Lions monotonicity conditions, which are recorded as follows.  
\begin{assumption}\label{assume-monotone}
The terminal map $g$ is monotone, namely, for all
$\rho,\sigma\in\cP^\circ(G)$,
\begin{equation}\label{terminal-mono}
 \ip{g(\rho)-g(\sigma)}{\rho-\sigma}\ge0.
\end{equation}
Moreover, the differential Lasry--Lions condition holds: for
$(\rho,p)\in\cP^\circ(G)\times\bbS^{d\times d}$ and
$(\eta,r)\in\bbR_0^d\times \bbS^{d\times d}$, 
\begin{equation}\label{differential-LL}
 \mathcal M_{\rho,p}(\eta,r)<0
 \qquad\text{whenever }(\eta,r)\ne(0,0),
\end{equation} where
\[
\begin{aligned}
 \mathcal M_{\rho,p}(\eta,r)
 :={}&\ip{D_\rho\bm H(\rho,p)[\eta]+D_p\bm H(\rho,p)[r]}{\eta}-\ip{D_\rho\bm B(\rho,p)[\eta]+D_p\bm B(\rho,p)[r]}{r}_E.\end{aligned}
\] 
\end{assumption}
\begin{remark}\label{remark1}
Condition
\eqref{differential-LL} implies that for
$(\rho,p),(\sigma,q)\in\cP^\circ(G)\times \bbS^{d\times d}$ with
$(\rho,p)\ne(\sigma,q)$,
\begin{equation}\label{LL}
 \ip{\bm H(\rho,p)-\bm H(\sigma,q)}{\rho-\sigma}
 -\ip{\bm B(\rho,p)-\bm B(\sigma,q)}{p-q}_E<0.
\end{equation}  
Moreover, condition \eqref{differential-LL} yields the following
quantitative form: for every compact
$K\Subset\cP^\circ(G)\times \bbS^{d\times d}$ there is $c_K>0$ such that
\begin{equation}\label{strong-LL-ineq}
\begin{aligned}
 &\ip{\bm H(\rho,p)-\bm H(\sigma,q)}{\rho-\sigma}
 -\ip{\bm B(\rho,p)-\bm B(\sigma,q)}{p-q}_E\le-c_K\bigl(
 \norm{\rho-\sigma}^2+\norm{p-q}_E^2\bigr)
\end{aligned}
\end{equation}
for all $(\rho,p),(\sigma,q)\in K$, 
whose proof is given in the appendix. 
The terminal condition \eqref{terminal-mono} together with \eqref{LL}
is sufficient for uniqueness of the solution of \eqref{mfg}.  The stronger differential condition
\eqref{differential-LL} is used below through \eqref{strong-LL-ineq} in the
stability and error analysis of numerical schemes.
\end{remark}

The following proposition gives the well-posedness and regularity estimates for the exact solution of the MFG \eqref{mfg}. The existence, uniqueness, and positivity assertions follow from
\cite[Theorem~1.2, Proposition~3.3, and Lemma~3.4]{GangboMunozWuZhang2026}. The temporal $\mathcal C^2$ regularity is obtained by differentiating
\eqref{mfg} along the exact trajectory, which remains in a compact
subset of the open simplex. The details are given in the appendix. 
\begin{proposition}
\label{cont-sol}
Let $\rho_0\in\cP_\epsilon$ for some $\epsilon\in(0,1/d)$ and let
Assumptions~\ref{continuous-data} and \ref{assume-monotone} hold.  Then \eqref{mfg} has a unique solution
$(\phi,\rho)\in\mathcal C^2([0,T];\bbR^d\times\cP^\circ(G))$. Moreover, there are
$\delta_*,C_*>0$, depending on $G,T,\epsilon$ and the problem data such that
\begin{equation}\label{cont-reg}
 \min_{0\le t\le T}\min_i\rho_i(t)\ge\delta_*,
 \qquad
 \norm{\phi}_{\mathcal C^2([0,T])}
 +\norm{\rho}_{\mathcal C^2([0,T])}\le C_*.
\end{equation}
\end{proposition}

Next, we introduce the potential subclass of \eqref{mfg}, for which the numerical scheme introduced in Section \ref{sec3} admits a variational formulation.
 Let $\cH:\cP^\circ(G)\times\bbS^{d\times d}\to\bbR$ be a Hamiltonian,
$\cF:\cP(G)\to\bbR$ a population potential, and
$\cU_T:\cP(G)\to\bbR$ a terminal potential.  The MFG is potential if
\begin{equation}\label{pot-coeff}
 \bm H(\rho,p)=D_\rho\cH(\rho,-p)+D_\rho\cF(\rho),
 \qquad
 \bm B(\rho,p)=-D_q\cH(\rho,-p),
 \qquad
 g=D_\rho\cU_T.
\end{equation}
In this case, 
the system also has a Wasserstein--Hamiltonian formulation; see \cite{AmbrosioGangbo2008,ChowLiZhou2020,CuiDieciZhou2022}.  Set
$S=-\phi$ and define
\begin{equation}\label{canonical-H}
 \mathscr K(\rho,S)
 =\cH(\rho,\gradG S)+\cF(\rho)
 -\ip{\gradG\rho}{\gradG S}_E.
\end{equation}
Then the potential MFG satisfies
$\dot\rho=D_S\mathscr K(\rho,S)$ and
$\dot S=-D_\rho\mathscr K(\rho,S)$.

\begin{assumption}
\label{pot-struct}
The Hamiltonian $\cH
\in\mathcal C^2(\cP^\circ(G)\times\bbS^{d\times d})$, and 
the functions $\cF,\cU_T\in\mathcal C^2(\cP(G))$. 
Moreover, $-\cF$ is strictly
convex and $\cU_T$ is convex on $\cP(G)$.   
For
$(\rho,q)\in\cP^\circ(G)\times\bbS^{d\times d}$,
$\eta\in\bbR_0^d$, and $r\in\bbS^{d\times d}$,
$\ip{D_{\rho\rho}^2\cH(\rho,q)\eta}{\eta}\le0$,
$\ip{D_{qq}^2\cH(\rho,q)r}{r}_E>0$ if $r\ne0$, and
$\ip{D^2\cF(\rho)\eta}{\eta}<0$ if $\eta\ne0$.
\end{assumption}

\begin{remark}\label{potential-monotonicity}
Assumption~\ref{pot-struct} implies
Assumption~\ref{assume-monotone}.  This follows from \eqref{pot-coeff}, the convexity
of $\cU_T$, the strict concavity of $\cF$,  and the concavity of $\mathcal H$ in $\rho$ variable together with its strict convexity in the edge variable. The verification is given in the appendix.  Hence
Proposition~\ref{cont-sol} applies to the potential system as a special case, under Assumptions \ref{continuous-data} and \ref{pot-struct}. 
\end{remark}  

We next present a concrete example of the potential MFG with a quadratic Hamiltonian and a quadratic potential; see \cite[(1.13)]{GangboMunozWuZhang2026}.    
\begin{example}\label{exam}
For $r,s>0$, let
$\theta(r,s)=(r-s)/(\log r-\log s)$ if $r\neq s$,  $\theta(r,r)=r$, and extend it by
$\theta(0,s)=\theta(s,0)=0$.  Write
$\theta_{ij}(\rho):=\theta(\rho_i,\rho_j)$ and define edgewise multiplication
by
$[\theta(\rho)q]_{ij}:=\theta_{ij}(\rho)q_{ij}$.  Define
\begin{equation}\label{quad-H}
 \cH_\theta(\rho,q)=\frac14\sum_{(i,j)\in E}
 \theta(\rho_i,\rho_j)q_{ij}^2.
\end{equation}

Choose a convex terminal potential $\cU_T\in\mathcal C^2(\cP(G))$ and let 
\begin{equation}\label{quad-data}
 \cF(\rho)=-\frac{\alpha}{2}\norm{\rho}^2,
 \qquad \alpha>0.
\end{equation}
Then the corresponding potential MFG is
\begin{equation}\label{quad-mfg}
\begin{cases}
\displaystyle
 \dot\phi_i=
 \frac12\sum_{j\sim i}\partial_1\theta(\rho_i,\rho_j)
 \bigl((\gradG\phi)_{ij}\bigr)^2
 -\alpha\rho_i-(\DeltaG\phi)_i,\\[3mm]
\displaystyle
 \dot\rho=\divG\bigl(\theta(\rho)\gradG\phi\bigr)+\DeltaG\rho, 
\end{cases}
\end{equation}
where $\partial_1\theta(r,s)$ denotes the derivative of
$\theta$ with respect to its first argument.  \eqref{quad-H}
and \eqref{quad-data} satisfy Assumptions~\ref{continuous-data} and
\ref{pot-struct}.  Hence, for $\rho_0\in\cP_\epsilon$,
Remark~\ref{potential-monotonicity} and Proposition~\ref{cont-sol} show
that \eqref{quad-mfg}, with $\rho(0)=\rho_0$ and
$\phi(T)=D_\rho\cU_T(\rho(T))$, has a unique solution satisfying
\eqref{cont-reg}.  The verification is given in the appendix. 
\end{example} 

\subsection{Comparison with Euclidean MFGs}
\label{sec2-compare} 
For comparison, consider the following general second-order MFG posed on the $r$-dimensional periodic Euclidean domain
$\mathbb T^r$: \[
\begin{cases}
 \partial_tu=\bm H_{\mathrm E}(m,\nabla_xu)-\nu\Delta_xu,\\
 \partial_tm=\operatorname{div}_x\bm B_{\mathrm E}(m,\nabla_xu)
 +\nu\Delta_xm,\\
 m(0,\cdot)=m_0,\qquad u(T,\cdot)=g_{\mathrm E}(m(T)).
\end{cases}
\]
Here, $r\ge 1$, $\Delta_x=\sum_{\ell=1}^r\partial_{x_\ell x_\ell}$ denotes the
Euclidean Laplacian, $\bm H_{\mathrm E}$ and $\bm B_{\mathrm E}$ denote the Hamiltonian and
the flux, respectively, and $g_{\mathrm E}$ is the terminal coupling.
The function $u$ is the value function, and $m$ is the nonnegative density
with unit mass.  This system has the same forward--backward structure as
\eqref{mfg}, with the spatial density $m$ replaced by the probability
vector $\rho$. Despite this similarity, the graph setting introduces
several distinctive analytical features, which are summarized below. 

First, the diffusion operators in the Euclidean and graph MFG systems
arise from different stochastic dynamics. 
Let $X_t$ satisfy
$\mathrm dX_t=\sqrt{2\nu}\,\mathrm dW_t$, where $W_t$ is an $r$-dimensional Brownian motion on $\mathbb T^r$, and let $Y_t$ be a continuous-time
Markov chain on $V$ whose transition
rates along the graph edges are given by the prescribed edge weights
$\omega_{ij}$. Their infinitesimal generators are
characterized by the following two identities:  
\[
\begin{aligned}
 &\lim_{h\downarrow0}\frac1h
 \mathbb E[\psi(X_{t+h})-\psi(X_t)\mid X_t=x]
 =\nu\Delta_x\psi(x),\quad \psi\in\mathcal C^2(\mathbb T^r),\\
 &\lim_{h\downarrow0}\frac1h
 \mathbb E[v(Y_{t+h})-v(Y_t)\mid Y_t=i]
 =\sum_{j\sim i}\omega_{ij}(v_j-v_i)=(\DeltaG v)_i,\quad v:V\to\bbR.
\end{aligned}
\]
Thus the Euclidean Laplacian is the generator of Brownian motion,
whereas the graph Laplacian is the generator of a continuous-time Markov
chain.
Consequently, the graph is itself the state space of the controlled
dynamics rather than a discretization of an underlying Euclidean domain.
Its topology and edge weights are therefore intrinsic components of the
model, which determine the admissible transitions and their rates; see \cite{ChowHuangLiZhou2012,Maas2011,ErbarMaas2012,GaoLiuLi2024}. Moreover, the connectedness of the graph plays a crucial role in the subsequent analysis through the graph Poincar\'e inequality and the timestep-uniform positivity estimate established below.

Second, unlike geometric boundaries in Euclidean domains, 
the relevant boundary in graph MFGs is the boundary of the probability simplex
$\partial\mathcal P(G)$, which is reached when one or more vertex
masses vanish.
For the logarithmic-mean model,
\[
\rho_i=0
\quad\Longrightarrow\quad
\theta(\rho_i,\rho_j)=0
\]
for every $j\sim i$, so the mobility becomes degenerate as the solution
approaches $\partial\mathcal P(G)$.
Consequently, the coefficients and their derivatives are no longer
uniformly controlled near the simplex boundary.
Since the convergence analysis relies on the uniform regularity of the nonlinear coefficients, it becomes essential to establish the
timestep-independent estimate
\[
\min_{0\le n\le N}\min_i\rho_i^n\ge\delta_T>0.
\]
Its proof relies essentially on the graph connectivity and the edge
structure, and it provides the key ingredient for the subsequent
stability and convergence analysis.
The corresponding continuous theory also requires quantitative
positivity estimates to prevent finite-time degeneration near
$\partial\mathcal P(G)$
\cite{GangboMunozWuZhang2026}.

Finally, potential MFGs in both settings admit variational
formulations; see
\cite{AchdouCamilliCapuzzoDolcetta2012,
BricenoAriasKaliseSilva2018,
BonnansLavignePfeiffer2023,
LavignePfeiffer2023}. 
However, the variational formulations lead to different analytical issues.  For the graph MFG considered here, the effective domain of the discrete action degenerates at the boundary of the probability simplex because of the logarithmic mobility.  This simplex-boundary degeneracy plays a direct role in the variational analysis of the discrete problem.

These features lead to a numerical analysis that is distinct from its Euclidean counterpart. In particular, preserving the
simplex interior is essential for maintaining the uniform regularity of
the nonlinear coefficients, and therefore underlies the  uniqueness, stability, and convergence analysis of the discrete scheme.

\section{Numerical scheme and main results}\label{sec3}
In this section, we introduce a numerical scheme for the general MFG
\eqref{mfg} and state its uniqueness, timestep-uniform estimates, and
first-order convergence. We then specialize to potential MFGs and establish
the variational characterization of the scheme. This variational characterization yields existence in the potential case and leads to an optimization-based numerical realization.

Let $t_n=n\tau$, $\tau=T/N$ with $N\in\mathbb N_+$, and $\rho_0\in\cP^\circ(G)$.  We propose the following numerical
scheme for \eqref{mfg}:  $\rho^0=\rho_0,\;
 \phi^N=g(\rho^N)$, 
\begin{equation}\label{mfg-scheme}
\begin{cases}
 \frac{\phi^{n+1}-\phi^n}{\tau}
 =\bm H(\rho^{n+1},\gradG\phi^n)-\DeltaG\phi^n,\\
 \frac{\rho^{n+1}-\rho^n}{\tau}
 =\divG\bm B(\rho^{n+1},\gradG\phi^n)+\DeltaG\rho^{n+1}.
\end{cases}
\end{equation}
Summing the density equation in \eqref{mfg-scheme} over the vertices gives
mass preservation:
$\sum_{i=1}^d\rho_i^{n+1}=\sum_{i=1}^d\rho_i^n$.  
The scheme is called time-staggered because both equations use
$(\rho^{n+1},\gradG\phi^n)$, which is the graph analogue of staggered implicit discretizations used for Euclidean MFGs and planning problems \cite{AchdouCapuzzoDolcetta2010,AchdouCamilliCapuzzoDolcetta2012,
AchdouCamilliCapuzzoDolcetta2013}. 
This placement is useful for several reasons. First, the two graph
Laplacians will cancel exactly in the
discrete fundamental identity, and the Lasry--Lions mechanism needed for uniqueness and the error estimate can be retained. Second, the density equation is
conservative and implicit in $\rho^{n+1}$, which is compatible with the
time-step-independent positivity estimate proved below. Moreover,  
in the
potential case, the scheme is
the KKT system for the discretization of the continuous potential action.

The first main result is stated below for the general MFG scheme. Unless otherwise stated, all constants in Theorem~\ref{general-results} may depend on the graph
$G$, the problem data, $\epsilon$, and $T$, but are independent of the timestep
$\tau$ and the number of time steps $N$.  
\begin{theorem}
\label{general-results}
Let $\rho^0\in\cP_\epsilon$ for some $\epsilon\in(0,\frac1d),$ and let Assumptions~\ref{continuous-data} and \ref{assume-monotone} hold.  Then \eqref{mfg-scheme} has at most one interior
solution. Moreover, every interior solution satisfies the following.  
 There exist $\delta_T,C_T>0$, independent of $N$ and $\tau$, such that
\begin{equation}\label{uniform-bounds}
 \min_{0\le n\le N}\min_i\rho_i^n\ge\delta_T,
 \qquad
 \max_{0\le n\le N}\Big(
 \norm{\phi^n}_{\ell^\infty}
 +\norm{\gradG\phi^n}_{\ell^\infty(E)}\Big)\le C_T.
\end{equation}
There exist $\tau_0>0$ and $\widetilde C_T>0$, independent
of $N$ and $\tau$, such that, for $0<\tau\le\tau_0$,
\begin{equation}\label{first-rate}
\begin{aligned}
 \max_{0\le n\le N}
 \big(\norm{\rho(t_n)-\rho^n}
 +\norm{\phi(t_n)-\phi^n}\big)+\Big(
 \tau\sum_{n=0}^{N-1}
 \norm{\gradG\phi(t_n)-\gradG\phi^n}_E^2
 \Big)^{1/2}
 \le \widetilde C_T\tau, 
\end{aligned}
\end{equation}where $(\phi,\rho)$ is the exact solution from
Proposition~\ref{cont-sol}.  
\end{theorem}

Theorem~\ref{general-results} establishes uniqueness,
timestep-uniform estimates, and first-order convergence
for every interior solution of the general scheme. 
We now turn to the potential case \eqref{pot-coeff}, where the variational characterization below also yields existence of an interior discrete solution.  
In the potential case, \eqref{mfg-scheme} becomes: 
$\rho^0=\rho_0$ and $\phi^N=D_\rho\cU_T(\rho^N)$, 
\begin{subequations}\label{potential-scheme}
\begin{empheq}[left=\empheqlbrace]{align}
 \frac{\phi^{n+1}-\phi^n}{\tau}
 &=D_\rho\cH(\rho^{n+1},-\gradG\phi^n)
 +D_\rho \cF(\rho^{n+1})-\DeltaG\phi^n,
 \label{scheme-phi}\\
 \frac{\rho^{n+1}-\rho^n}{\tau}
 &=-\divG D_q\cH(\rho^{n+1},-\gradG\phi^n)
 +\DeltaG\rho^{n+1}.
 \label{scheme-rho}
\end{empheq}
\end{subequations}  
With $S^n=-\phi^n$ and $\mathscr K$ defined in \eqref{canonical-H},
\eqref{potential-scheme} is the symplectic Euler discretization
\[
 \frac{\rho^{n+1}-\rho^n}{\tau}
 =D_S\mathscr K(\rho^{n+1},S^n),
 \qquad
 \frac{S^{n+1}-S^n}{\tau}
 =-D_\rho\mathscr K(\rho^{n+1},S^n).
\]
Thus the Hamiltonian structure also explains the staggered evaluation at
$(\rho^{n+1},S^n)$.

To identify the proposed scheme with a convex optimization problem,
we use
the Legendre duality between the Hamiltonian and the Lagrangian.  
Accordingly, we introduce the Lagrangian $\cL:\cP(G)\times\bbS^{d\times d}\to\bbR\cup\{+\infty\}$, whose Legendre transform in the edge variable is the Hamiltonian
$\cH$ in \eqref{pot-coeff}, i.e., 
$\cH(\rho,q)=\sup_{w\in\bbS^{d\times d}}
\{\ip{w}{q}_E-\cL(\rho,w)\}$. 
 For sequences $\rho=(\rho^n)_{n=0}^N$ and
$m=(m^{n+1})_{n=0}^{N-1}$,  impose
\begin{equation}\label{continuity}
 {\Gamma^n(\rho,m):=\frac{\rho^{n+1}-\rho^n}{\tau}+\divG m^{n+1}=0,}
 \qquad 0\le n\le N-1.
\end{equation}

Set $w^{n+1}:=m^{n+1}+\gradG\rho^{n+1}$ and define the discrete action
$$\mathcal A_\tau(\rho,m) :=\tau\sum_{n=0}^{N-1} \bigl[\cL(\rho^{n+1},w^{n+1})-\cF(\rho^{n+1})\bigr] +\cU_T(\rho^N).$$ This is the right-endpoint time discretization of
the continuous potential action introduced in \cite[Section~1.2]{GangboMunozWuZhang2026}.  
We consider the discrete variational problem 
\begin{equation}\label{min-prob}
 \min\Bigl\{
 \mathcal A_\tau(\rho,m):
 \rho\in\cP(G)^{N+1},\
 m\in(\bbS^{d\times d})^N,\
 \rho^0=\rho_0,\
 {\Gamma^n(\rho,m)=0},\ 0\le n<N
 \Bigr\}. 
\end{equation} 
Problem~\eqref{min-prob} is close in form to dynamic graph optimal transport because it minimizes a convex action over fluxes subject to a linear continuity equation; see \cite{ErbarRumpfSchmitzerSimon2020}.  The key difference is that this is not formulated to
compute a transport distance between two prescribed endpoint measures. 
Instead, its KKT conditions characterize
exactly the discrete potential graph MFG scheme.

To derive its first-order optimality conditions, we introduce the
constraint Lagrangian 
\begin{equation}\label{kkt-cond}
 \mathfrak L_\tau(\rho,m,\phi)
 =
 \mathcal A_\tau(\rho,m)
 {-\tau\sum_{n=0}^{N-1}\ip{\phi^n}{\Gamma^n(\rho,m)}}.
\end{equation} 
 The multiplier $\phi^n$ is attached to the $n$th
continuity constraint and is chosen so that the first variation of \(\mathfrak L_\tau\) with respect to \((\rho,m)\) vanishes in every primal direction. Set
$ 
 {\Gamma(\rho,m)}
 :=
 {\bigl(\Gamma^0(\rho,m),\ldots,\Gamma^{N-1}(\rho,m)\bigr)}
 \in(\bbR_0^d)^N.
$ 
A triple $(\rho,m,\phi)$ satisfies the
\textit{Karush--Kuhn--Tucker (KKT) conditions} if
\begin{equation}\label{kkt-system}
 {\Gamma(\rho,m)=0},\qquad
 D_{(\rho,m)}\mathfrak L_\tau(\rho,m,\phi)
 [\eta,\zeta]=0
\end{equation}
for every
$(\eta,\zeta)\in(\bbR_0^d)^N
\times(\bbS^{d\times d})^N$.
The first condition is equivalent to $D_\phi\mathfrak L_\tau=0$.
Thus the KKT multiplier $\phi$ is the dual variable associated with the discrete continuity constraint \eqref{continuity}.  For the discrete variational problem, we impose the following assumptions
on $\cL$, adapted from \cite[Section~4]{GangboMunozWuZhang2026}. 
\begin{assumption}\label{variational-data}
The Lagrangian $\cL$ is proper, lower semicontinuous, and jointly convex in
$(\rho,w)$, and strictly convex in $w$ whenever it is finite.  Its coercivity
is uniform in the density: for every $A\in\bbR$ there is $R_A>0$ such that
\begin{equation}\label{coercivity}
 \rho\in\cP(G),\quad \cL(\rho,w)\le A
 \quad\text{implies}\quad \norm{w}_E\le R_A.
\end{equation}
 Moreover,
\begin{equation}\label{zero-action}
 \cL(\rho,0)<\infty
 \qquad\text{for every }\rho\in\cP^\circ(G).
\end{equation}
The effective domain of $\cL$ satisfies: 
\begin{equation}\label{action-boundary}
 \cL(\rho,w)<\infty,\quad \rho_i=0
 \quad\text{implies}\quad
 w_{ij}=0\quad\text{for every }j\sim i.
\end{equation}
In addition, $\cL\in\mathcal C^1(\cP^\circ(G)\times\bbS^{d\times d})$, and $\cL$ is locally uniformly
superlinear in $w$: for every compact set $K\Subset\cP^\circ(G)$,
\begin{equation}\label{superlinear}
 \lim_{R\to\infty}
 \inf_{\substack{\rho\in K\\ \norm{w}_E\ge R}}
 \frac{\cL(\rho,w)}{\norm{w}_E}=+\infty.
\end{equation}
\end{assumption}
For the logarithmic-mean quadratic model in
Example~\ref{exam}, the Legendre dual of $\cH_\theta$ is
\begin{equation}\label{quad-L}
 \cL_\theta(\rho,w)=\frac14\sum_{(i,j)\in E}
 \frac{w_{ij}^2}{\theta(\rho_i,\rho_j)},
\end{equation}
with the lower-semicontinuous convention that a summand is zero if both
$w_{ij}$ and $\theta_{ij}(\rho)$ vanish and is $+\infty$ if
$\theta_{ij}(\rho)=0$ but $w_{ij}\ne0$. 
The appendix verifies Assumption~\ref{variational-data} for
$\cL_\theta$.
The following theorem establishes
the equivalence between the constrained minimization problem and the potential
MFG scheme.  It proves that the discrete action has a unique interior
minimizer, identifies the normalized KKT multiplier with the value variable
in \eqref{potential-scheme}, and shows conversely that every interior solution
of the scheme generates this minimizer.

\begin{theorem}
\label{potential-sol}
Let $\rho_0\in\cP^\circ(G)$, and let
Assumptions~\ref{pot-struct} and \ref{variational-data} hold.  Then
\begin{itemize}
\item[(i)] The
minimization problem \eqref{min-prob} has a unique minimizer
$(\rho_*,m_*)$, and $\rho_*^n\in\cP^\circ(G)$ for every $n$.

\item[(ii)] There are unique  Lagrange multipliers
$\phi^0,\ldots,\phi^{N-1}\in\bbR_0^d$ such that
$(\rho_*,m_*,\phi)$ satisfies \eqref{kkt-system}.  Set
$\phi^N=D_\rho\cU_T(\rho_*^N)$.  There are unique constants
$a_0,\ldots,a_{N-1}$ such that, with $a_N=0$ and
$\widehat\phi^n=\phi^n+a_n\bm1,0\le n\le N$, the pair
$(\widehat\phi,\rho_*)$ solves \eqref{potential-scheme}.  

\item[(iii)] Conversely, every interior solution
$(\phi,\rho)$ of \eqref{potential-scheme} generates the unique minimizer of \eqref{min-prob} by
\begin{equation}\label{optimal-flux}
 w^{n+1}=D_q\cH(\rho^{n+1},-\gradG\phi^n),
 \qquad
 m^{n+1}=w^{n+1}-\gradG\rho^{n+1}.
\end{equation}
\end{itemize}
In particular, \eqref{potential-scheme} has an interior solution.
If, in addition, Assumption~\ref{continuous-data} holds and $\rho_0\in\cP_\epsilon$, 
then Remark~\ref{potential-monotonicity} supplies Assumption~\ref{assume-monotone}, so Theorem~\ref{general-results} applies and this solution satisfies \eqref{uniform-bounds} and \eqref{first-rate}. 
\end{theorem}

We mention that the variational characterization of the scheme is  specific to the potential case.  
A general pair $(\bm H,\bm B)$ need not be the derivative of a
single scalar functional, so such a convex minimization problem is not
available without an additional integrability condition. This
is also standard for potential MFGs in Euclidean spaces; see \cite{AchdouCamilliCapuzzoDolcetta2012,
BonnansLavignePfeiffer2023,LavignePfeiffer2023}.

For the logarithmic-mean quadratic model, the verification in the appendix shows that all assumptions of Theorem~\ref{potential-sol} and Theorem~\ref{general-results} hold.  Hence its discrete solution exists, is unique and uniformly interior, and satisfies the first-order estimate \eqref{first-rate}.

\section{Proof of  Theorem~\ref{general-results}}\label{sec4} 
This section proves Theorem~\ref{general-results}, whose proof is divided into two parts.
In Section \ref{sec4.1}, we establish uniqueness together with the timestep-independent
interior and stability estimates.
In Section \ref{sec4.2}, we combine these estimates with a consistency argument and the
discrete fundamental identity to derive the first-order convergence. Throughout this section,
we work under the hypotheses  of Theorem~\ref{general-results} and let
$(\phi^n,\rho^n)_{n=0}^N$ be a solution of
\eqref{mfg-scheme}.  

\subsection{Uniqueness and uniform estimates}\label{sec4.1}

This subsection proves the uniqueness and uniform bounds \eqref{uniform-bounds} in
Theorem~\ref{general-results}.
We begin with uniqueness of the numerical solution in the open simplex.
\begin{lemma}
\label{disc-unique}
Under Assumptions~\ref{continuous-data} and
\ref{assume-monotone}, the scheme 
\eqref{mfg-scheme} has at most one solution satisfying
$\rho^n\in\cP^\circ(G)$ for $0\le n\le N$.
\end{lemma}

\begin{proof}
Let $(\phi,\rho)$ and $(\widetilde\phi,\widetilde\rho)$ be two such
solutions.  Set
$\delta\phi^n=\phi^n-\widetilde\phi^n$,
$\delta\rho^n=\rho^n-\widetilde\rho^n$, and
$\delta p^n=\gradG\delta\phi^n$.

\smallskip
\noindent\emph{Step 1. Discrete fundamental identity.}
Define
\[
\begin{aligned}
 \delta\bm H^n
 :=
 \bm H(\rho^{n+1},\gradG\phi^n)
 -\bm H(\widetilde\rho^{n+1},\gradG\widetilde\phi^n),\;\;
 \delta\bm B^n
 :=
 \bm B(\rho^{n+1},\gradG\phi^n)
 -\bm B(\widetilde\rho^{n+1},\gradG\widetilde\phi^n).
\end{aligned}
\]
We first prove
\begin{equation}\label{disc-identity}
 \ip{\delta\phi^{n+1}}{\delta\rho^{n+1}}
 -\ip{\delta\phi^n}{\delta\rho^n}
 =\tau\ip{
 \delta \bm H^n}
 {\delta\rho^{n+1}}-\tau\ip{\delta
 \bm B^n}
 {\delta p^n}_E,
\end{equation}
for $0\le n<N$.
Testing the difference of the value equation in
\eqref{mfg-scheme} with $\delta\rho^{n+1}$, and the density equation with
$\delta\phi^n$ gives
\begin{align*}
 \frac1\tau\ip{\delta\phi^{n+1}-\delta\phi^n}{\delta\rho^{n+1}}
 &=\ip{\delta\bm H^n}{\delta\rho^{n+1}}
 -\ip{\DeltaG\delta\phi^n}{\delta\rho^{n+1}},\\ \frac1\tau\ip{\delta\rho^{n+1}-\delta\rho^n}{\delta\phi^n}
 &=\ip{\divG\delta\bm B^n}{\delta\phi^n}
 +\ip{\DeltaG\delta\rho^{n+1}}{\delta\phi^n}.
\end{align*}
Adding these two identities and using 
\begin{equation*}
 -\ip{\DeltaG\delta\phi^n}{\delta\rho^{n+1}}
 +\ip{\delta\phi^n}{\DeltaG\delta\rho^{n+1}}=0,\quad \ip{\divG\delta\bm B^n}{\delta\phi^n}
=-\ip{\delta\bm B^n}{\delta p^n}_E, 
\end{equation*}
proves
\eqref{disc-identity}.

\smallskip
\noindent\emph{Step 2. Uniqueness.}
Assumption~\ref{assume-monotone} gives
\eqref{terminal-mono}, while \eqref{LL} follows from
\eqref{differential-LL}. 
For $0\le n<N$, set
$ 
 M_n:=\ip{\delta\bm H^n}{\delta\rho^{n+1}}
 -\ip{\delta\bm B^n}{\delta p^n}_E.
$ 
Condition~\eqref{LL} gives $M_n<0$ whenever
$(\rho^{n+1},\gradG\phi^n)\ne
(\widetilde\rho^{n+1},\gradG\widetilde\phi^n)$, while $M_n=0$ when these
two pairs are equal.  Hence $M_n\le0$, with equality if and only if
$\rho^{n+1}=\widetilde\rho^{n+1}$ and
$\gradG\phi^n=\gradG\widetilde\phi^n$.  Summing \eqref{disc-identity} from $n=0$ to $N-1$, and using
$\delta\rho^0=0$ and  $\ip{\delta\phi^N}{\delta\rho^N}
=\ip{g(\rho^N)-g(\widetilde\rho^N)}
{\rho^N-\widetilde\rho^N}\ge0$, we obtain
$$0\le\ip{\delta\phi^N}{\delta\rho^N}
=\tau\sum_{n=0}^{N-1}M_n\le0.$$
Hence 
\[
 \delta\rho^{n+1}=0,
 \qquad
 \delta p^n=\gradG\phi^n-\gradG\widetilde\phi^n=0,
 \qquad 0\le n<N.
\]
For every edge $(i,j)\in E$,
\[
 0=(\gradG\delta\phi^n)_{ij}
 =\sqrt{\omega_{ij}}(\delta\phi_i^n-\delta\phi_j^n).
\]
Since $\omega_{ij}>0$, this gives
$\delta\phi_i^n=\delta\phi_j^n$ at adjacent vertices.  Because $G$ is
connected, any two vertices $i$ and $j$ can be joined by a path
$i=i_0\sim i_1\sim\cdots\sim i_m=j$.  Along this path we have 
$\delta\phi_i^n=\delta\phi_{i_1}^n=\cdots=\delta\phi_j^n$.
Thus there exist $b_0,\ldots,b_{N-1}\in\mathbb R$ such that
$ 
\delta\phi^n=b_n\mathbf1,\; 0\le n<N.
$ 
Moreover,
$\delta\rho^N=0$ and the terminal condition give
$\delta\phi^N=g(\rho^N)-g(\widetilde\rho^N)=0$. Set $b_N=0$.
Subtracting the two
value equations in \eqref{mfg-scheme} gives
\[
 \frac{\delta\phi^{n+1}-\delta\phi^n}{\tau}
 =
 \bm H(\rho^{n+1},\gradG\phi^n)
 -\bm H(\widetilde\rho^{n+1},\gradG\widetilde\phi^n)
 -\DeltaG\delta\phi^n.
\]
Since $\delta\rho^{n+1}=0$ and
$\gradG\phi^n=\gradG\widetilde\phi^n$, the difference of the two $\bm H$ terms is zero.  Moreover,
\[
 \DeltaG\delta\phi^n
 =\DeltaG(b_n\bm1)=b_n\DeltaG\bm1=0.
\]
Consequently,
\[
 \frac{b_{n+1}-b_n}{\tau}\bm1=0,
 \qquad 0\le n<N,
\]which yields 
$b_{n+1}=b_n$.  Starting from
$b_N=0$ and proceeding backward gives $b_n=0$ for every $0\le n\le N$. This finishes the uniqueness proof.
\end{proof}

We now prove estimates that are uniform in the time step. 
\begin{lemma}
\label{val-grad}
Let Assumption~\ref{continuous-data} hold and let
$\rho^0\in\cP_\epsilon$.  Suppose that $(\phi^n,\rho^n)_{n=0}^N$ solves 
\eqref{mfg-scheme}.  Set $p^n=\gradG\phi^n$.  Then
\begin{equation}\label{one-sol-dual}
\begin{aligned}
 \ip{\phi^{n+1}}{\rho^{n+1}}
 -\ip{\phi^n}{\rho^n}
 &=\tau\ip{\bm H(\rho^{n+1},p^n)}{\rho^{n+1}}-\tau\ip{\bm B(\rho^{n+1},p^n)}{p^n}_E.
\end{aligned}
\end{equation}
Moreover, 
\begin{equation}\label{duality-bound}
 \ip{\phi^{n+1}}{\rho^{n+1}}
 -\ip{\phi^n}{\rho^n}\le C_0\tau.
\end{equation}
There is a constant $C_\phi>0$ such that 
\begin{equation}\label{val-bound}
 \max_{0\le n\le N}\norm{\phi^n}_{\ell^\infty}\le C_\phi,\quad \norm{\gradG\phi^n}_{\ell^\infty(E)}
 \le R_\phi:=2\sqrt{\omega_{\max}}C_\phi,
\end{equation}
where $C_\phi$ depends only on
$d,T,\epsilon,C_0,C_H,C_g$.  
\end{lemma}
\begin{proof}
Test the value equation in \eqref{mfg-scheme} with $\rho^{n+1}$ and
the density equation with $\phi^n$.
By \eqref{graph-ibp}, the two Laplacian terms cancel and the
divergence term becomes $-\ip{\bm B}{p^n}_E$. 
This gives
\eqref{one-sol-dual}.  The first inequality in
\eqref{data-bounds} gives \eqref{duality-bound}.

Set $M_n=\max_i\phi_i^n$.  Let $i_n$ be a vertex at which
$\phi^n$ attains its maximum.  Then \[(\Delta_G\phi^n)_{i_n}
=
\sum_{j\sim i_n}\omega_{i_nj}
(\phi_j^n-\phi_{i_n}^n)
\le 0.\]
The $i_n$th component of the value equation in
\eqref{mfg-scheme} and the bound
$\bm H_{i_n}\ge-C_H$ give
$\phi_{i_n}^{n+1}-M_n =\tau\bigl[\bm H_{i_n}(\rho^{n+1},p^n) -(\DeltaG\phi^n)_{i_n}\bigr] \ge-C_H\tau$.
Since $\phi_{i_n}^{n+1}\le M_{n+1}$, we have $M_n\le M_{n+1}+C_H\tau
.$ Iterating this inequality and using the terminal condition \(\phi^N=g(\rho^N)\), we obtain \begin{equation}\label{val-upper}
  M_n
\le M_N+C_H(N-n)\tau
\le C_g+C_HT
=:M_+, 
\end{equation} where we recall that $C_g$ is given in Assumption \ref{continuous-data}. 
Write $J_n=\ip{\phi^n}{\rho^n}$.  Summing
\eqref{duality-bound} gives
$J_0\ge J_N-C_0T\ge-C_g-C_0T=:-A$.
For every vertex $i$, use $\rho_i^0\ge\epsilon$ and
$\phi_j^0\le M_+$ to obtain
$$\rho_i^0\phi_i^0 =J_0-\sum_{j\ne i}\rho_j^0\phi_j^0 \ge-A-M_+, \quad \phi_i^0\ge-\frac{A+M_+}{\epsilon}.$$
Thus, with $S_n=\sum_{i=1}^d \phi_i^n$,
$S_0\ge-\frac{d(A+M_+)}{\epsilon}$.
Summing the value equation in \eqref{mfg-scheme} over the vertices and using
$\sum_{i=1}^d(\DeltaG\phi^n)_i=0$ yields
$S_{n+1}-S_n =\tau\sum_i\bm H_i(\rho^{n+1},p^n) \ge-dC_H\tau$.
Therefore
$$
S_n
=S_0+\sum_{k=0}^{n-1}(S_{k+1}-S_k)\ge S_0-dC_Hn\tau\ge-\frac{d(A+M_+)}{\epsilon}-dC_HT.
$$ This, together with \(\sum_{j\ne i}\phi_j^n\le(d-1)M_+,\) yields 
\begin{equation*}
\begin{aligned}
 \phi_i^n
 =S_n-\sum_{j\ne i}\phi_j^n\ge-\frac{d(A+M_+)}{\epsilon}
 -dC_HT-(d-1)M_+.
\end{aligned}
\end{equation*}
Together with \eqref{val-upper}, this proves the first inequality in 
\eqref{val-bound}.  Finally,
$|(\gradG\phi^n)_{ij}| =\sqrt{\omega_{ij}}|\phi_i^n-\phi_j^n| \le2\sqrt{\omega_{\max}}C_\phi$,
which finishes the proof. 
\end{proof}

\begin{lemma}
\label{q-interior}
Let Assumption~\ref{continuous-data} hold. 
Suppose that
$\rho^0\in\cP^\circ(G)$ and $\rho^n\in\cP(G)$ for every $n$, and 
$\norm{\gradG\phi^n}_{\ell^\infty(E)}\le R$.  If
$\underline\rho^n=\min_i\rho_i^n$, then there is a constant $C_R$, independent of
$n,N,\tau$, such that
\begin{equation}\label{interior-step}
 \underline\rho^{n+1}\ge\frac{\underline\rho^n}{1+C_R\tau}.
\end{equation}
Consequently,
\begin{equation}\label{global-interiority}
 \underline\rho^n\ge e^{-C_Rt_n}\underline\rho^0.
\end{equation}
\end{lemma} 
\begin{proof}
We first justify strict positivity.  If $\rho^n>0$ and
$\rho^{n+1}$ has a zero component, define
$Z=\{i:\rho_i^{n+1}=0\}, \; P=V\setminus Z$.
Mass conservation gives $P\ne\varnothing$.  Since $G$ is connected and
$Z\ne\varnothing$, there is an edge $(i,j)$ with $i\in Z$ and $j\in P$.
For every $k\sim i$, \eqref{boundary-B_revise} gives
$\bm B_{ik}(\rho^{n+1},p^n)=0$, where we set $p^n=\gradG\phi^n$.   Therefore
$\frac{\rho_i^{n+1}-\rho_i^n}{\tau} =0+(\DeltaG\rho^{n+1})_i =\sum_{k\sim i}\omega_{ik}\rho_k^{n+1}>0$.
The left-hand side is $-\rho_i^n/\tau<0$, whereas the
right-hand side is strictly positive, which is a contradiction.    Induction
from $\rho^0>0$ proves strict positivity.

For each fixed time level $n$, choose a vertex
$i_0:=i_0(n)$ at which $\rho^{n+1}$ attains its minimum:
$\rho_{i_0}^{n+1}=\underline\rho^{n+1}$.
For each neighbor $j\sim i_0$, set
$u_j=\frac{\rho_j^{n+1}}{\underline\rho^{n+1}}\ge1$.
The $i_0$ component of the density equation is
\begin{equation*}
\frac{\rho_{i_0}^{n+1}-\rho_{i_0}^n}{\tau}
 =(\divG\bm B(\rho^{n+1},p^n))_{i_0}
 +(\DeltaG\rho^{n+1})_{i_0}.
\end{equation*}
Thus every edge $(i_0,j)$ contributes one transport term and one diffusion
term.  By \eqref{graph-laplacian} the diffusion contribution of this edge is
\begin{equation*}
\omega_{i_0j}(\rho_j^{n+1}-\rho_{i_0}^{n+1})
 =\omega_{i_0j}\underline\rho^{n+1}(u_j-1)\ge 0.
\end{equation*}
By 
\eqref{graph-divergence} and \eqref{bounded-mobility}, the transport contribution of this edge
is
\begin{align*}
-\sqrt{\omega_{i_0j}}\, \bm B_{i_0j}(\rho^{n+1},p^n)
\ge-\sqrt{\omega_{i_0j}}
 |\bm B_{i_0j}(\rho^{n+1},p^n)|\ge-\sqrt{\omega_{i_0j}}\,
 \underline\rho^{n+1}(1+u_j)h_R(u_j).
\end{align*}
Therefore
\begin{align}
\omega_{i_0j}(\rho_j^{n+1}-\rho_{i_0}^{n+1})
-\sqrt{\omega_{i_0j}}\,\bm B_{i_0j}(\rho^{n+1},p^n)
 &\ge
 \underline\rho^{n+1}\bigl[
 \omega_{i_0j}(u_j-1)
 -\sqrt{\omega_{i_0j}}(1+u_j)h_R(u_j)
 \bigr]\notag\\
 &\ge   -\underline\rho^{n+1}
 \bigl[
 \sqrt{\omega_{i_0j}}(1+u_j)h_R(u_j)
 -\omega_{i_0j}(u_j-1)
 \bigr]_+,\label{edge-min}
\end{align}
where we write $a_+=\max\{a,0\}$ 
for a real number $a$, and use 
$-a\ge-a_+$. 
Define 
$C_R= \max_{i\in V}\sum_{j\sim i} \sup_{u\ge1} \bigl[ \sqrt{\omega_{ij}}(1+u)h_R(u)-\omega_{ij}(u-1) \bigr]_+$.
This constant is finite.  
Indeed, on every compact interval,  the term $[\sqrt{\omega_{ij}}(1+u)h_R(u)-\omega_{ij}(u-1) ]$ is bounded. On the other hand, since $h_R$ is bounded and tends to zero as $u\to\infty$, 
there exists $U_R\ge1$
such that
$\sqrt{\omega_{ij}}h_R(u)\le\omega_{ij}/2$ for every edge $(i,j)$ and every
$u\ge U_R$.   For $u\ge U_R$, 
\begin{align*}
 &\sqrt{\omega_{ij}}(1+u)h_R(u)-\omega_{ij}(u-1)\le\frac{\omega_{ij}}2(1+u)-\omega_{ij}(u-1)
 =\frac{3\omega_{ij}}2-\frac{\omega_{ij}}2u,
\end{align*}
which is bounded above on $[U_R,\infty)$.  Since the graph has finitely many
edges, we have $C_R<\infty$. 
Summing the edge estimates \eqref{edge-min} over
all neighbors of $i_0$ gives
\begin{align*}
 \frac{\rho_{i_0}^{n+1}-\rho_{i_0}^n}{\tau}
 &=\sum_{j\sim i_0}\bigl[
 \omega_{i_0j}(\rho_j^{n+1}-\rho_{i_0}^{n+1})
 -\sqrt{\omega_{i_0j}}\,\bm B_{i_0j}(\rho^{n+1},p^n)\bigr]\\
 &\ge
 -\underline\rho^{n+1}
 \sum_{j\sim i_0}
 \bigl[
 \sqrt{\omega_{i_0j}}(1+u_j)h_R(u_j)
 -\omega_{i_0j}(u_j-1)
 \bigr]_+\ge -C_R\underline\rho^{n+1}.
\end{align*}
Multiplying by $\tau$ and using $\rho_{i_0}^{n+1}=\underline\rho^{n+1}$ gives
$ 
\underline\rho^{n+1}\ge \rho_{i_0}^n-C_R\tau\underline\rho^{n+1}
 \ge \underline\rho^n-C_R\tau\underline\rho^{n+1}.
$ 
Rearranging gives \eqref{interior-step}.  Iteration and
$(1+C_R\tau)^n\le e^{C_Rt_n}$ prove
\eqref{global-interiority}.
\end{proof}
\begin{proof}[Proof of uniqueness and \eqref{uniform-bounds}
in Theorem~\ref{general-results}]
Uniqueness follows from Lemma~\ref{disc-unique}.
Lemma~\ref{val-grad} gives the bound for $\phi^n$ and
$\gradG\phi^n$, uniformly in $N$ and $\tau$.  Apply
Lemma~\ref{q-interior} with $R=R_\phi$, 
where $R_\phi=2\sqrt{\omega_{\max}}C_\phi$ is given in \eqref{val-bound}. 
It follows that
$ 
 \min_i\rho_i^n
 \ge\epsilon e^{-C_{R_\phi}t_n}
 \ge\epsilon e^{-C_{R_\phi}T}=:\delta_T>0.
$ 
This proves the density estimate.  Together with Lemma~\ref{val-grad}, it gives
\eqref{uniform-bounds}.  
\end{proof}

\subsection{First-order convergence}\label{sec4.2}
In this subsection, we prove the error estimate \eqref{first-rate} in 
Theorem~\ref{general-results}. The proof relies on the regularity of the exact
solution (see Proposition~\ref{cont-sol}) and the timestep-uniform estimates in  \eqref{uniform-bounds} proved in
Section \ref{sec4.1}. 
We begin by estimating the consistency residual obtained by inserting the exact solution into the discrete scheme. 
\begin{lemma}
Let the conditions of Theorem~\ref{general-results} hold. Set
$\Phi^n=\phi(t_n)$, $R^n=\rho(t_n)$.  Then
\begin{equation}\label{exact-res}
\begin{aligned}
 \frac{\Phi^{n+1}-\Phi^n}{\tau}
 &=\bm H(R^{n+1},\gradG\Phi^n)-\DeltaG\Phi^n+r_\phi^n,
 \\
 \frac{R^{n+1}-R^n}{\tau}
 &=\divG\bm B(R^{n+1},\gradG\Phi^n)
 +\DeltaG R^{n+1}+r_\rho^n, 
\end{aligned}
\end{equation}
where the residuals $r^n_{\phi},r^n_{\rho}$ satisfy 
\begin{equation}\label{local-res}
 \max_n\bigl(\norm{r_\phi^n}+\norm{r_\rho^n}\bigr)\le C\tau,\quad
 \sum_{i=1}^d r_{\rho,i}^n=0,\; 0\le n<N.
\end{equation}
\end{lemma} 
\begin{proof}
Write $P^n=\gradG\Phi^n$.  The fundamental theorem of
calculus and the exact value equation at $t_n$ give
\[
\begin{aligned}
 &\quad \frac{\Phi^{n+1}-\Phi^n}{\tau}
 =\dot\phi(t_n)
 +\frac1\tau\int_{t_n}^{t_{n+1}}
 \bigl[\dot\phi(s)-\dot\phi(t_n)\bigr]\,\mathrm{d}s\\
 &=\bm H(R^n,P^n)-\DeltaG\Phi^n
 +\frac1\tau\int_{t_n}^{t_{n+1}}
 \bigl[\dot\phi(s)-\dot\phi(t_n)\bigr]\,\mathrm{d}s=\bm H(R^{n+1},P^n)-\DeltaG\Phi^n+r_\phi^n,
\end{aligned}
\]
where
\begin{align}
 r_\phi^n
 :=\frac1\tau\int_{t_n}^{t_{n+1}}
 \bigl[\dot\phi(s)-\dot\phi(t_n)\bigr]\,\mathrm{d}s
 +\bm H(R^n,P^n)-\bm H(R^{n+1},P^n).
 \label{phi-defect-int}
\end{align}
Let $L_H$ be a Lipschitz constant of $\bm H$ on a compact neighborhood of the
exact trajectory.  Since
\begin{equation}\label{time-increments}
 \norm{R^{n+1}-R^n}\le\tau\norm{\dot\rho}_{L^\infty(0,T)},
 \qquad
 \norm{P^{n+1}-P^n}_E
 \le C\tau\norm{\dot\phi}_{L^\infty(0,T)},
\end{equation} 
 \eqref{phi-defect-int} yields
$\norm{r_\phi^n} \le\frac\tau2\norm{\ddot\phi}_{L^\infty(0,T)} +L_H\tau\norm{\dot\rho}_{L^\infty(0,T)}$.
For the density equation, we use the exact equation at $t_{n+1}$ and obtain
\[
\begin{aligned}
 \frac{R^{n+1}-R^n}{\tau}
 &=\dot\rho(t_{n+1})
 +\frac1\tau\int_{t_n}^{t_{n+1}}
 \bigl[\dot\rho(s)-\dot\rho(t_{n+1})\bigr]\,\mathrm{d}s\\
 &=\divG\bm B(R^{n+1},P^{n+1})+\DeltaG R^{n+1}
 +\frac1\tau\int_{t_n}^{t_{n+1}}
 \bigl[\dot\rho(s)-\dot\rho(t_{n+1})\bigr]\,\mathrm{d}s\\
 &=\divG\bm B(R^{n+1},P^n)+\DeltaG R^{n+1}+r_\rho^n,
\end{aligned}
\]
where
\begin{align}
 r_\rho^n
 :=\frac1\tau\int_{t_n}^{t_{n+1}}
 \bigl[\dot\rho(s)-\dot\rho(t_{n+1})\bigr]\,\mathrm{d}s+\divG\bigl[
 \bm B(R^{n+1},P^{n+1})-\bm B(R^{n+1},P^n)\bigr].
 \label{rho-defect-int}
\end{align}
Both terms in
\eqref{rho-defect-int} have zero sum.
Let $L_B$ be a Lipschitz constant of $\bm B$ on the same compact.  Equations
\eqref{time-increments} and
\eqref{rho-defect-int} imply
$ 
\norm{r_\rho^n}
 \le\frac\tau2\norm{\ddot\rho}_{L^\infty(0,T)}
 +CL_B\tau\norm{\dot\phi}_{L^\infty(0,T)}.
$ This proves \eqref{local-res}.
\end{proof}
The following graph Poincar\'e inequality will be used repeatedly:
\begin{equation}\label{Pincare-ineq}
 \norm{v-\overline v\bm1}
 \le C_P\norm{\gradG v}_E,
 \qquad
 \overline v=\frac1d\sum_i v_i,
\end{equation}
which follows from connectedness of $G$.
Indeed, connectedness gives
$\ker(\gradG)=\operatorname{span}\{\bm1\}$; hence $\gradG$ is injective on
$\bbR_0^d$, and equivalence of norms on this finite-dimensional space yields
\eqref{Pincare-ineq}.
\begin{proof}[Proof of
Theorem~\ref{general-results}] The uniqueness and timestep-uniform estimates have been established in
Section~\ref{sec4.1}.
It remains to prove the first-order error estimate \eqref{first-rate}.  
Set
$e_\phi^n=\phi(t_n)-\phi^n, \; e_\rho^n=\rho(t_n)-\rho^n, \; e_p^n=\gradG e_\phi^n$.
The exact trajectory is contained in a compact subset of
$\cP^\circ(G)\times \bbS^{d\times d}$ by Proposition~\ref{cont-sol}. The uniform estimates \eqref{uniform-bounds}
show that the numerical trajectory remains in a compact subset,
independent of $\tau$.
Hence, for all sufficiently small $\tau$, both trajectories are
contained in a common compact set 
$
K_*\subset
\mathcal P^\circ(G)\times\mathbb S^{d\times d},
$ 
which is independent of $\tau$.
Hence the coefficients
$\bm H$ and $\bm B$ are uniformly Lipschitz on $K_*$. Moreover, the differential Lasry--Lions condition
\eqref{differential-LL}, together with Remark \ref{remark1}, yields the quantitative
estimate \eqref{strong-LL-ineq} on $K_*$ with a constant
$c_{K_*}>0$.  
 Let $r_\phi^n,r_\rho^n$ be the residuals in
\eqref{exact-res}.  
Define
\begin{align*}
 \delta\bm H^n
 &:=\bm H(R^{n+1},\gradG\Phi^n)
 -\bm H(\rho^{n+1},\gradG\phi^n),
 \;
 \delta\bm B^n
 :=\bm B(R^{n+1},\gradG\Phi^n)
 -\bm B(\rho^{n+1},\gradG\phi^n).
\end{align*}
Subtracting the numerical scheme from
\eqref{exact-res} gives
\begin{subequations}
\begin{align}
 \frac{e_\phi^{n+1}-e_\phi^n}{\tau}
 &=\delta\bm H^n-\DeltaG e_\phi^n+r_\phi^n,
 \label{err-phi}\\
 \frac{e_\rho^{n+1}-e_\rho^n}{\tau}
 &=\divG\delta\bm B^n+\DeltaG e_\rho^{n+1}+r_\rho^n.
 \label{rho-error}
\end{align}
\end{subequations}
Taking the inner product of \eqref{err-phi} with
$e_\rho^{n+1}$ and of \eqref{rho-error} with
$e_\phi^n$, and adding the resulting identities, we obtain
\begin{align}
 \ip{e_\phi^{n+1}}{e_\rho^{n+1}}
 -\ip{e_\phi^n}{e_\rho^n}&=\tau\ip{\delta\bm H^n}{e_\rho^{n+1}}
 -\tau\ip{\delta\bm B^n}{e_p^n}_E
 +\tau\ip{r_\phi^n}{e_\rho^{n+1}}
 +\tau\ip{e_\phi^n}{r_\rho^n}\notag\\
 &
\le-c_{K_*}\tau\norm{e_\rho^{n+1}}^2
-c_{K_*}\tau\norm{e_p^n}_E^2
 +\tau\ip{r_\phi^n}{e_\rho^{n+1}}
 +\tau\ip{e_\phi^n}{r_\rho^n},\label{error-identity}
\end{align} 
where we use 
\eqref{strong-LL-ineq} on the compact set $K_*$. 
By $\sum_i r_{\rho,i}^n=0$ and the graph Poincar\'e inequality \eqref{Pincare-ineq},
\begin{equation*}
\bigl|\ip{e_\phi^n}{r_\rho^n}\bigr|
 =\bigl|\ip{e_\phi^n-\overline e_\phi^n\bm1}{r_\rho^n}\bigr|
 \le C_P\norm{e_p^n}_E\norm{r_\rho^n}.
\end{equation*}
Moreover, $e_\rho^0=0$, while \eqref{terminal-mono} gives
\begin{equation*}
 \ip{e_\phi^N}{e_\rho^N}
 =\ip{g(\rho(T))-g(\rho^N)}
 {\rho(T)-\rho^N}\ge0.
\end{equation*}
Sum \eqref{error-identity} from $n=0$ to $N-1$.  The left-hand
side telescopes.  Using $e_\rho^0=0$ and
the preceding terminal inequality, we obtain
\begin{align*}
 &c_{K_*} \tau\sum_{n=0}^{N-1}\norm{e_\rho^{n+1}}^2
 +c_{K_*}\tau\sum_{n=0}^{N-1}\norm{e_p^n}_E^2\le\tau\sum_{n=0}^{N-1}
 \norm{r_\phi^n}\norm{e_\rho^{n+1}}
 +C_P\tau\sum_{n=0}^{N-1}
 \norm{r_\rho^n}\norm{e_p^n}_E,
\end{align*}
which together with Young's inequality and  
\eqref{local-res} gives
\begin{equation}\label{error-energy-bound}
 \tau\sum_{n=0}^{N-1}
 \bigl(\norm{e_\rho^{n+1}}^2+\norm{e_p^n}_E^2\bigr)
 \le C_T\tau^2.
\end{equation}

It remains to estimate bounds $ (\max\limits_{0\le n\le N}\norm{e_\phi^n}\vee \max\limits_{0\le n\le N}\norm{e_\rho^n})\le C_T\tau$.  
Taking the inner product of \eqref{rho-error} with
$e_\rho^{n+1}$, and using \eqref{graph-ibp} and the polarization identity $2\ip{a-b}{a}=\norm{a}^2-\norm{b}^2+\norm{a-b}^2$ for $a,b\in\mathbb R^d,$ we obtain
\begin{align}\label{equation1}
 &\frac1{2\tau}\bigl(
 \norm{e_\rho^{n+1}}^2-\norm{e_\rho^n}^2
 +\norm{e_\rho^{n+1}-e_\rho^n}^2\bigr)
 +\norm{\gradG e_\rho^{n+1}}_E^2=-\ip{\delta\bm B^n}{\gradG e_\rho^{n+1}}_E
 +\ip{r_\rho^n}{e_\rho^{n+1}}. 
\end{align}
For the term $
-\ip{\delta\bm B^n}{\gradG e_\rho^{n+1}}_E$, we have 
\begin{align*}
|-\ip{\delta\bm B^n}{\gradG e_\rho^{n+1}}_E|&=|\ip{\divG\delta\bm B^n}{e_\rho^{n+1}}|\leq  \norm{\divG\delta\bm B^n}\|e^{n+1}_{\rho}\|
\\&\le C\norm{\delta\bm B^n}_E\|e^{n+1}_{\rho}\|
\le
C\bigl(
\norm{e_\rho^{n+1}}+\norm{e_p^n}_E
\bigr)\|e^{n+1}_{\rho}\|,
\end{align*}
where we use \eqref{graph-ibp}, the boundedness of the operator $\divG$, and the Lipschitz continuity of $\bm B$ on $K_*$. Dropping the nonnegative terms on the left-hand side of \eqref{equation1}, we obtain  
\begin{align*}
 \frac1{2\tau}\bigl(\norm{e_\rho^{n+1}}^2
 -\norm{e_\rho^n}^2\bigr)
 &\le C\bigl(\norm{e_\rho^{n+1}}
 +\norm{e_p^n}_E\bigr)\norm{e_\rho^{n+1}}+\norm{r_\rho^n}\norm{e_\rho^{n+1}}.  
\end{align*} Apply
Young's inequality
to the products
$\norm{e_p^n}_E\norm{e_\rho^{n+1}}$ and
$\norm{r_\rho^n}\norm{e_\rho^{n+1}}$, and then move the resulting multiples of
$\norm{e_\rho^{n+1}}^2$ to the left-hand side of the inequality.  It therefore gives a constant $C_1$, independent of $n$ and
$\tau$, such that
$(1-C_1\tau)\norm{e_\rho^{n+1}}^2 \le\norm{e_\rho^n}^2 +C_1\tau\norm{e_p^n}_E^2 +C_1\tau\norm{r_\rho^n}^2$.
Choose $\tau_0$ so that $C_1\tau_0\le1/2$.  Since
$(1-C_1\tau)^{-1}\le1+2C_1\tau$ for $\tau\le\tau_0$, we obtain
$$\norm{e_\rho^{n+1}}^2 \le(1+C\tau)\norm{e_\rho^n}^2 +C\tau\norm{e_p^n}_E^2 +C\tau\norm{r_\rho^n}^2.$$
Iterating this recursion and using $e_\rho^0=0$ yields
\begin{align*}
 \norm{e_\rho^n}^2
 &\le C\tau\sum_{k=0}^{n-1}(1+C\tau)^{n-1-k}
 \bigl(\norm{e_p^k}_E^2+\norm{r_\rho^k}^2\bigr)\le Ce^{CT}\tau\sum_{k=0}^{N-1}
 \bigl(\norm{e_p^k}_E^2+\norm{r_\rho^k}^2\bigr)
 \le C_T\tau^2,
\end{align*}
where the last inequality follows from \eqref{error-energy-bound} and \eqref{local-res}.  Consequently,
\begin{equation}\label{rho-max}
 \max_{0\le n\le N}\norm{e_\rho^n}\le C_T\tau.
\end{equation}

Finally, the value error equation is equivalent to
\begin{equation}\label{phi-resolve}
 (I-\tau\DeltaG)e_\phi^n
 =e_\phi^{n+1}-\tau\delta\bm H^n-\tau r_\phi^n.
\end{equation}
Since $-\DeltaG$ is symmetric positive semidefinite, every eigenvalue of
$I-\tau\DeltaG$ is at least one.  Hence
$\norm{(I-\tau\DeltaG)^{-1}}_{\mathcal L(\bbR^d)}\le1$.
Applying this resolvent to \eqref{phi-resolve} and using
the Lipschitz property of $\bm H$ gives
\begin{equation}\label{phi-rec}
 \norm{e_\phi^n}
 \le\norm{e_\phi^{n+1}}
 +L\tau\bigl(\norm{e_\rho^{n+1}}+\norm{e_p^n}_E\bigr)
 +\tau\norm{r_\phi^n}.
\end{equation}
Because $g$ is Lipschitz on the compact simplex,
$ 
 \norm{e_\phi^N}
 =\norm{g(R^N)-g(\rho^N)}
 \le L_g\norm{e_\rho^N}\le C_T\tau.
$ 
Iterating \eqref{phi-rec} backward from $N$ to $n$ gives
\begin{align}
 \norm{e_\phi^n}
 &\le\norm{e_\phi^N}
 +L\tau\sum_{k=n}^{N-1}\norm{e_\rho^{k+1}}
 +L\tau\sum_{k=n}^{N-1}\norm{e_p^k}_E
 +\tau\sum_{k=n}^{N-1}\norm{r_\phi^k}\leq C_T\tau,
 \label{phi-back}
\end{align}where we use 
 \eqref{local-res}, \eqref{rho-max}, the 
Cauchy--Schwarz inequality, and \eqref{error-energy-bound}.  
This, together with 
equations \eqref{error-energy-bound} and 
\eqref{rho-max} proves \eqref{first-rate}.
\end{proof}

\section{Proof of Theorem~\ref{potential-sol}}\label{sec5}
This section proves Theorem~\ref{potential-sol}.  In Subsection~\ref{var-part}, we prove the existence,
uniqueness, and interiority of the minimizer and state the Legendre relations
needed for the KKT system.  In Subsection~\ref{kkt-part}, we identify the KKT
conditions with the potential MFG scheme and complete the proof of
Theorem~\ref{potential-sol}.

\subsection{Variational characterization}\label{var-part}

We begin by constructing an interior admissible path with finite action. 
\begin{lemma}
\label{heat-competitor} 
Let Assumptions~\ref{pot-struct} and
\ref{variational-data} hold, and let $\rho_0\in\cP^\circ(G)$. 
Then there exists at
least one pair $(\rho,m)$ satisfying \eqref{continuity}, $\rho^0=\rho_0$,
$\rho^n\in\cP^{\circ}(G)$, and $\mathcal A_\tau(\rho,m)<\infty$.
\end{lemma}

\begin{proof}
It suffices to exhibit a single admissible path with finite action. 
 Starting from the prescribed initial density $\rho^0=\rho_0$, define
recursively
\begin{equation}\label{heat-path}
 (I-\tau\DeltaG)\rho^{n+1}=\rho^n,
 \qquad
 m^{n+1}=-\gradG\rho^{n+1}.
\end{equation}
 Since \(-\Delta_G\) is symmetric positive semidefinite because
$\ip{-\DeltaG u}{u}=\norm{\gradG u}_E^2\ge0$, its eigenvalues $\lambda$ are nonnegative. Hence the eigenvalues of \(I-\tau\Delta_G=I+\tau(-\Delta_G)\) are \(1+\tau\lambda\ge1\), and the matrix \(I-\tau\Delta_G\) is invertible. 
The first equation in \eqref{heat-path} is equivalent to
$({\rho^{n+1}-\rho^n})/{\tau}-\DeltaG\rho^{n+1}=0$.
Since $m^{n+1}=-\gradG\rho^{n+1}$ and $\DeltaG=\divG\gradG$, we have 
\begin{align}\label{eq-1}
\frac{\rho^{n+1}-\rho^n}{\tau} +\divG m^{n+1} = \frac{\rho^{n+1}-\rho^n}{\tau} -\DeltaG\rho^{n+1}=0.
\end{align}
Thus the constraint \eqref{continuity} holds and
$w^{n+1}=m^{n+1}+\gradG\rho^{n+1}=0$.

We next check that the densities $\rho^n$ remain in the open simplex $\mathcal P^{\circ}(G)$.   Assume
$\rho^n>0$, and let $i_0$ be a vertex where $\rho^{n+1}$ attains its minimum.
For every neighbor $j$ of $i_0$,
$\rho_j^{n+1}-\rho_{i_0}^{n+1}\ge0$.  Therefore
$(\DeltaG\rho^{n+1})_{i_0} =\sum_{j\sim i_0}\omega_{i_0j} (\rho_j^{n+1}-\rho_{i_0}^{n+1})\ge0$.
Using \eqref{eq-1} at the vertex $i_0$ gives
$\rho_{i_0}^{n+1} =\rho_{i_0}^n+\tau(\DeltaG\rho^{n+1})_{i_0} \ge\rho_{i_0}^n>0$.
Since the minimum component of $\rho^{n+1}$ is positive, every component of
$\rho^{n+1}$ is positive.  Equation~\eqref{eq-1} also preserves the sum of
the components.  Induction therefore proves
$\rho^n\in\cP^\circ(G)$ for all $n$.

Finally, $w^{n+1}=0$ and \eqref{zero-action} give
$\cL(\rho^{n+1},w^{n+1})=\cL(\rho^{n+1},0)<\infty$ for every $n$.
The remaining terms $-\cF(\rho^{n+1})$ and $\cU_T(\rho^N)$ are finite because
$\cF$ and $\cU_T$ are continuous on the compact simplex.  Hence
$\mathcal A_\tau(\rho,m)<\infty$. We finish the proof. 
\end{proof}

Lemma \ref{heat-competitor} constructs an admissible pair with finite action. Hence the infimum in \eqref{min-prob} is strictly below $+\infty$; the lower bound needed to exclude $-\infty$ is proved next. 
We now establish existence and uniqueness of the minimizer.  
\begin{lemma}
\label{minimizer-lemma}
Let $\rho_0\in\cP^\circ(G)$.  Under
Assumptions~\ref{pot-struct} and \ref{variational-data}, the minimization
problem \eqref{min-prob} admits a unique minimizer in
$\mathcal P(G)^{N+1}\times (\mathbb S^{d\times d})^N$.
\end{lemma}

\begin{proof}
Let $\mathcal K_\tau$ be the feasible set, defined by \[
 \mathcal K_\tau
 :=
 \Bigl\{
 (\rho,m)\in
 \cP(G)^{N+1}\times(\bbS^{d\times d})^N:
 \rho^0=\rho_0,\ 
 \frac{\rho^{n+1}-\rho^n}{\tau}
 +\divG m^{n+1}=0,\ 
 0\le n<N
 \Bigr\}.
\]
We write $\rho=(\rho^0,\ldots,\rho^N)$ and $m=(m^1,\ldots,m^N)$, where $\rho^n\in\cP(G)$ and $m^n\in \bbS^{d\times d}$.  
Put
$I_\tau=\inf_{(\rho,m)\in\mathcal K_\tau}\mathcal A_\tau(\rho,m)$. We first show that \(I_\tau\) is finite. On one hand, 
by Lemma~\ref{heat-competitor}, there is
$(\bar\rho,\bar m)\in\mathcal K_\tau$ with
$\mathcal A_\tau(\bar\rho,\bar m)<\infty$.  Hence  
$I_\tau\le\mathcal A_\tau(\bar\rho,\bar m)<\infty$. We next
prove the lower bound $I_\tau>-\infty$.  
The continuous functions $-\cF$ and $\cU_T$ are bounded below on the compact
simplex.  The Lagrangian $\cL$ is also bounded below.  Indeed, take $R_0$ from
\eqref{coercivity} with $A=0$.  Then
$\cL(\rho,w)>0$ for $\|w\|_E>R_0$.  We claim that $\cL$ is bounded below on
the compact set
$K_0=\cP(G)\times\{w:\|w\|_E\le R_0\}$.  Otherwise, there would be
$(\rho_k,w_k)\in K_0$ such that $\cL(\rho_k,w_k)\to-\infty$.
After extracting a subsequence, $(\rho_k,w_k)\to(\rho,w)\in K_0$, and lower
semicontinuity would give
$\cL(\rho,w)\le\liminf_k\cL(\rho_k,w_k)=-\infty$, which is contrary to properness.
Hence $\cL$ has a finite lower bound on $K_0$, and it is positive outside
$K_0$.  Thus the action is bounded below on
$\mathcal K_\tau$, and $I_\tau>-\infty$.

Choose a minimizing sequence
$(\rho_k,m_k)\subset\mathcal K_\tau$ such that
$\mathcal A_\tau(\rho_k,m_k)\le I_\tau+1/k$.
Here
$\rho_k=(\rho_k^0,\ldots,\rho_k^N),\; m_k=(m_k^1,\ldots,m_k^N)$,
where the subscript $k$ labels the element of the minimizing sequence and the
superscript $n$ labels the time level.  
We have $\mathcal A_\tau(\rho_k,m_k)\le C$ for all $k$.  Set
$w_k^{n+1}=m_k^{n+1}+\gradG\rho_k^{n+1}$.
Let $\ell_*$ be a lower bound for $\cL$, and let $b_F,b_T$ be lower bounds
for $-\cF$ and $\cU_T$ on $\cP(G)$.  For every $0\le n<N$,
$$\tau\cL(\rho_k^{n+1},w_k^{n+1})
\le C-\tau(N-1)\ell_*-N\tau b_F-b_T.$$
Thus 
$\cL(\rho_k^{n+1},w_k^{n+1})\leq C_{\tau}$ uniformly in $k$.
By
\eqref{coercivity}, we obtain 
$\max_{0\le n<N}\norm{w_k^{n+1}}_E\le C_\tau$, and thus 
$m_k^{n+1}=w_k^{n+1}-\gradG\rho_k^{n+1}$ is bounded because   $\|\gradG\rho_k^n\|_{E}\leq C$, 
where $\norm{\cdot}_E$ is given in \eqref{infty-E}.  
We may therefore
extract a subsequence, still denoted by $(\rho_k,m_k)$, such that
$\rho_k^n\longrightarrow\rho^n, \, m_k^{n+1}\longrightarrow m^{n+1}$
for every time level $n$.  Passing to the limit gives 
\begin{equation*}
0=\lim_{k\to\infty}
 (\frac{\rho_k^{n+1}-\rho_k^n}{\tau}
 +\divG m_k^{n+1})
 =\frac{\rho^{n+1}-\rho^n}{\tau}
 +\divG m^{n+1}.
\end{equation*}
Also $\rho^0=\rho_0$ and $\rho^n\in\cP(G)$, since $\cP(G)$ is closed.  Thus the limit path is feasible.  
By 
$w_k^{n+1}\to m^{n+1}+\gradG\rho^{n+1}$, lower semicontinuity of $\cL$ and
continuity of $\cF,\cU_T$ give
$\mathcal A_\tau(\rho,m) \le\liminf_{k\to\infty}\mathcal A_\tau(\rho_k,m_k)$.
Thus $(\rho,m)$ is a minimizer.

For uniqueness, let $(\rho,m)$ and $(\widetilde\rho,\widetilde m)$ be two
minimizers and let $w=m+\gradG\rho$ and
$\widetilde w=\widetilde m+\gradG\widetilde\rho$.  Define their midpoint by
$
\widehat\rho^n:=\frac{\rho^n+\widetilde\rho^n}{2},\;
 \widehat m^{n+1}:=\frac{m^{n+1}+\widetilde m^{n+1}}{2}.
$ 
Since $\cP(G)$ is convex, $\widehat\rho^n\in\cP(G)$ and
$\widehat\rho^0=\rho_0$.  Moreover,
\begin{align*}
 \frac{\widehat\rho^{n+1}-\widehat\rho^n}{\tau}
 +\divG\widehat m^{n+1}
 &=
 \frac12\Bigl(
 \frac{\rho^{n+1}-\rho^n}{\tau}+\divG m^{n+1}
 \Bigr)+
 \frac12\Bigl(
 \frac{\widetilde\rho^{n+1}-\widetilde\rho^n}{\tau}
 +\divG\widetilde m^{n+1}
 \Bigr)=0.
\end{align*}
Thus $(\widehat\rho,\widehat m)$ is feasible and $
 \widehat w^{n+1}
 =\widehat m^{n+1}+\gradG\widehat\rho^{n+1}
 =\frac{w^{n+1}+\widetilde w^{n+1}}{2}.
$ 
By the convexity assumptions, for $0\le n<N$,
\begin{align*}
 &\cL(\widehat\rho^{n+1},\widehat w^{n+1})
\le\frac12\cL(\rho^{n+1},w^{n+1})
 +\frac12\cL(\widetilde\rho^{n+1},\widetilde w^{n+1}),\\
 &-\cF(\widehat\rho^{n+1})
 \le-\frac12\cF(\rho^{n+1})
 -\frac12\cF(\widetilde\rho^{n+1}),\qquad 
 \cU_T(\widehat\rho^N)
 \le\frac12\cU_T(\rho^N)
 +\frac12\cU_T(\widetilde\rho^N).
\end{align*}
Suppose that $\rho^k\ne\widetilde\rho^k$ for some $1\le k\le N$.
Since $-\cF$ is strictly convex, the second inequality is strict at
$n+1=k$.  Summing the preceding inequalities gives
\[
 \mathcal A_\tau(\widehat\rho,\widehat m)
 <\frac12\mathcal A_\tau(\rho,m)
 +\frac12\mathcal A_\tau(\widetilde\rho,\widetilde m)
 =I_\tau,
\]
contrary to the definition of $I_\tau$.  Hence
$\rho^n=\widetilde\rho^n$ for every $0\le n\le N$. 
It remains to compare the flux variables $w,\widetilde w$.  If
$w^k\ne\widetilde w^k$ for some $1\le k\le N$, and both
$\cL(\rho^k,w^k)$ and $\cL(\rho^k,\widetilde w^k)$ are finite, then strict
convexity in the second variable gives
\[
 \cL\Bigl(\rho^k,\frac{w^k+\widetilde w^k}{2}\Bigr)
 <\frac12\cL(\rho^k,w^k)
 +\frac12\cL(\rho^k,\widetilde w^k).
\]
All the density terms are now identical, and the remaining Lagrangian terms
are convex.  Thus the midpoint would again have action strictly smaller than
$I_\tau$, a contradiction.  Therefore
$w^n=\widetilde w^n$ for every $1\le n\le N$, and
$ 
 m^n=w^n-\gradG\rho^n
 =\widetilde w^n-\gradG\widetilde\rho^n
 =\widetilde m^n,\; 1\le n\le N.
$ 
\end{proof}

The following lemma shows that the boundary condition \eqref{action-boundary} on $\cL$ prevents a
finite-action path from reaching the boundary of the simplex.
\begin{lemma}
\label{finite-positivity}
Let Assumption~\ref{variational-data} hold and let
$\rho^0\in\cP^\circ(G)$.   If $(\rho,m)$ satisfies
\eqref{continuity}, $\rho^n\in\cP(G)$ for every $n$ and has finite action,
then
$\rho^n\in\cP^\circ(G)$ for every $0\le n\le N$.
\end{lemma}

\begin{proof}
Suppose to the contrary that the path reaches the boundary of the simplex.
Let $n+1$ be the first time layer for which some component vanishes.  Thus,
there is a vertex $i$ such that
$\rho_i^{n+1}=0, \; \rho^k\in\cP^\circ(G)\;\text{for }0\le k\le n$.
In particular $\rho_i^n>0$.
Recall that $w^{n+1}=m^{n+1}+\gradG\rho^{n+1}$.
Since the action is finite,
$\cL(\rho^{n+1},w^{n+1})<\infty$.  Because
$\rho_i^{n+1}=0$, condition \eqref{action-boundary} therefore gives
$w_{ij}^{n+1}=0 \text{ for every }j\sim i$.
Using this identity, we obtain
\begin{equation*}
 m_{ij}^{n+1}
 =-(\gradG\rho^{n+1})_{ij}
 =-\sqrt{\omega_{ij}}(\rho_i^{n+1}-\rho_j^{n+1})
 =\sqrt{\omega_{ij}}\rho_j^{n+1}.
\end{equation*}
Now evaluate the divergence at the vertex $i$.  By
\eqref{graph-divergence}, we have 
$ 
 (\divG m^{n+1})_i
 =-\sum_{j\sim i}\omega_{ij}\rho_j^{n+1}\le0.
$ 
Substituting $\rho_i^{n+1}=0$ to the $i$th component of the discrete continuity equation
$\frac{\rho_i^{n+1}-\rho_i^n}{\tau} +(\divG m^{n+1})_i=0$ 
 gives 
$0 =-\frac{\rho_i^n}{\tau} -\sum_{j\sim i}\omega_{ij}\rho_j^{n+1}<0$.
This contradiction proves that no
component can vanish at any time layer.
\end{proof}

To identify the KKT equations with the discrete MFG system, we use the
following standard consequence of convex duality and Danskin's theorem; see
\cite[Section~26]{Rockafellar1970} and
\cite[Theorem~4.13 and Remark~4.14]{BonnansShapiro2000}.
\begin{lemma}\label{legendre}
Let Assumption~\ref{variational-data} hold.  For every
$(\rho,q)\in\cP^\circ(G)\times\bbS^{d\times d}$, the supremum
in the Legendre formula 
\[
\cH(\rho,q)= \sup_{w\in\bbS^{d\times d}}
 \bigl\{\ip{w}{q}_E-\cL(\rho,w)\bigr\}
\]
is attained at a unique point $W(\rho,q)$. 
The map $W$ is continuous, $\cH$ is $\mathcal C^1$, and
\[
 q=D_w\cL(\rho,w)
 \quad\text{if and only if}\quad
 w=W(\rho,q)=D_q\cH(\rho,q).
\]
If these equivalent relations hold, then
$D_\rho\cL(\rho,w)=-D_\rho\cH(\rho,q)$. 
\end{lemma}

\subsection{Proof of Theorem \ref{potential-sol}}\label{kkt-part}
We now present the proof of Theorem \ref{potential-sol}, by using the following finite-dimensional Lagrange multiplier theorem; see e.g. 
\cite[Chapter~12]{NocedalWright2006}.
\begin{lemma}\label{lagrange-mul-thm} If $\mathsf X$ and $\mathsf Y$ are
finite-dimensional inner-product spaces, {$\mathsf U\subset\mathsf X$
is open, $J:\mathsf U\to\bbR$ and
$\Gamma:\mathsf U\to\mathsf Y$} are $\mathcal C^1$, and
{$x_*\in\mathsf U$} is a local minimizer
of $J$ subject to ${\Gamma(x)=0}$, then surjectivity of
${D\Gamma(x_*):\mathsf X\to\mathsf Y}$ implies that there exists a unique
$\lambda\in\mathsf Y$ such that
$DJ(x_*)[\delta x]-\ip{\lambda}{{D\Gamma(x_*)[\delta x]}}_{\mathsf Y}=0$ for every
$\delta x\in\mathsf X$.  Equivalently, $x_*$ is a critical point of
$J(x)-\ip{\lambda}{{\Gamma(x)}}_{\mathsf Y}$.
\end{lemma}
\begin{proof}[Proof of Theorem~\ref{potential-sol}] \textit{Proof of (i).}
{Lemma~\ref{minimizer-lemma}} gives a unique minimizer, which we write as
$ 
\rho_*=(\rho_*^0,\ldots,\rho_*^N),\;
m_*=(m_*^1,\ldots,m_*^N).
$ 
The path constructed in Lemma~\ref{heat-competitor} is admissible and has finite action, so the minimum of $\mathcal A_\tau$ is finite.  
Lemma~\ref{finite-positivity}
therefore gives $\rho_*^n\in\cP^\circ(G)$ for every $n$. 

\textit{Proof of (ii).}
{This proof is split into three steps. First, we apply Lemma~\ref{lagrange-mul-thm} to obtain the unique
Lagrange multipliers for the discrete continuity constraints and hence the
KKT system.  Then we present the stationarity equations for $\mathfrak L_{\tau}$. We finally fix the additive constants to recover the unique value
variable in \eqref{potential-scheme}.

\textit{Step 1: Apply Lemma~\ref{lagrange-mul-thm}.} Put
$\overline\rho=d^{-1}\bm1$ and write
$y^n=\rho^n-\overline\rho\in\bbR_0^d$ for $1\le n\le N$.  Define
\[
 \mathsf X=(\bbR_0^d)^N\times(\bbS^{d\times d})^N,
 \qquad
 \mathsf Y=(\bbR_0^d)^N,
\]
and use the coordinates
\[
 x=(y^1,\ldots,y^N,m^1,\ldots,m^N)\in\mathsf X,
 \qquad
 \rho^0(x)=\rho_0,\qquad \rho^n(x)=\overline\rho+y^n.
\]
The density coordinates range in the open set
\[
 \mathsf U=\bigl\{x\in\mathsf X:
 \overline\rho+y^n\in\cP^\circ(G),\ 1\le n\le N\bigr\}.
\]
Part (i) gives $x_*\in\mathsf U$, where
\[
 x_*=(\rho_*^1-\overline\rho,\ldots,
 \rho_*^N-\overline\rho,m_*^1,\ldots,m_*^N).
\]
On $\mathsf U$, set
\[
 J(x)=\mathcal A_\tau(\rho(x),m),
 \qquad
 \Gamma(x)=(\Gamma^0(x),\ldots,\Gamma^{N-1}(x)),
\]
where
\[
 \Gamma^n(x)=\frac{\rho^{n+1}(x)-\rho^n(x)}{\tau}
 +\divG m^{n+1},\qquad 0\le n<N.
\]
Each $\Gamma^n(x)$ belongs to $\bbR_0^d$, so
$\Gamma:\mathsf U\to\mathsf Y$.  {Assumptions~\ref{pot-struct}
and \ref{variational-data} imply that $J$ is $\mathcal C^1$ on
$\mathsf U$, while $\Gamma$ is affine and hence $\mathcal C^1$.}  Moreover,
$x_*$ minimizes $J$ subject to $\Gamma(x)=0$.

It remains to verify the surjectivity required by
Lemma~\ref{lagrange-mul-thm}.  For
\[
 \delta x=(\eta^1,\ldots,\eta^N,
 \zeta^1,\ldots,\zeta^N)\in\mathsf X,
 \qquad \eta^0=0,
\]
the affine form of $\Gamma$ gives
\[
 {D\Gamma^n(x_*)[\delta x]}
 =\frac{\eta^{n+1}-\eta^n}{\tau}
 +\divG\zeta^{n+1},\qquad 0\le n<N.
\]
Given an arbitrary
$q=(q^0,\ldots,q^{N-1})\in\mathsf Y$, choose
\[
 \zeta^{n+1}=0,\qquad
 \eta^0=0,\qquad
 \eta^{n+1}=\eta^n+\tau q^n,\qquad 0\le n<N.
\]
Then $\eta^n\in\bbR_0^d$ and
$D\Gamma(x_*)[\delta x]=q$.  Hence
$D\Gamma(x_*):\mathsf X\to\mathsf Y$ is surjective.}

{We have now verified all the hypotheses of
Lemma~\ref{lagrange-mul-thm}.  Therefore, there exists a unique}
$\lambda=(\lambda^0,\ldots,\lambda^{N-1})\in\mathsf Y=(\bbR_0^d)^N$ such that 
\begin{equation}\label{mult-id}
 DJ(x_*)[\delta x]
 -\sum_{n=0}^{N-1}\ip{\lambda^n}
 {{D\Gamma^n(x_*)[\delta x]}}=0
 \qquad\text{for every }\delta x\in\mathsf X.
\end{equation}
Define $\phi^n:=\lambda^n/\tau\in\bbR_0^d$.  {Then
\eqref{mult-id} becomes}
{
\[
 0=D_{(\rho,m)}
 \mathfrak L_\tau(\rho_*,m_*,\phi)[\eta,\zeta]
\]}
for every
$(\eta,\zeta)\in(\bbR_0^d)^N
\times(\bbS^{d\times d})^N$, {where we recall that
$\mathfrak L_\tau$ is defined in \eqref{kkt-cond}.}
For arbitrary constants $a_n$, replacing $\phi^n$ by
$\phi^n+a_n\bm1$ does not change $\mathfrak L_\tau$, since
$ {\Gamma^n(x)}\in\bbR_0^d$ and
$\ip{\bm1}{{\Gamma^n(x)}}=0$.  Thus the
multiplier can be equivalently viewed as a class in
$\bbR^d/\operatorname{span}\{\bm1\}$.  We will uniquely determine the constants $a_n$ in \textit{Step 3} based on the density stationarity
equations derived in \textit{Step 2} and the given terminal condition. 

\textit{Step 2: We now deduce the stationarity equations.}
We begin by taking the variation with respect to $m^{n+1}$. 
Fix $0\le n<N$ and $\zeta\in\bbS^{d\times d}$.  Recall that the
optimal flux is $m_*=(m_*^1,\ldots,m_*^N)$.  
Stationarity with $m^{n+1}$ for $\mathfrak L_{\tau}$ therefore gives
\begin{align*}
 0
 &=\tau\ip{D_w\cL(\rho_*^{n+1},w_*^{n+1})}{\zeta}_E
 -\tau\ip{\phi^n}{\divG\zeta}=\tau\ip{D_w\cL(\rho_*^{n+1},w_*^{n+1})
 +\gradG\phi^n}{\zeta}_E.
\end{align*}
Since $\zeta$ is arbitrary, this yields 
\begin{equation}\label{kkt-momentum}
 D_w\cL(\rho_*^{n+1},w_*^{n+1})=-\gradG\phi^n. 
\end{equation} 
Set $q=-\gradG\phi^n$.  Equation~\eqref{kkt-momentum} gives
$q=D_w\cL(\rho_*^{n+1},w_*^{n+1})$.  Since
$\cL(\rho_*^{n+1},\cdot)$ is convex, for every
$z\in\bbS^{d\times d}$,
\begin{align*}
 \cL(\rho_*^{n+1},z)
 &\ge\cL(\rho_*^{n+1},w_*^{n+1})
 +\ip{q}{z-w_*^{n+1}}_E,
 \end{align*}
 and thus \begin{align*}
 \ip{z}{q}_E-\cL(\rho_*^{n+1},z)
 &\le\ip{w_*^{n+1}}{q}_E
 -\cL(\rho_*^{n+1},w_*^{n+1}),
\end{align*} which means that  $w_*^{n+1}$ is the unique maximizer in the definition of
$\cH(\rho_*^{n+1},q)$. Applying  Lemma~\ref{legendre} gives
$w_*^{n+1}=D_q\cH(\rho_*^{n+1},-\gradG\phi^n)$.
Since $m_*^{n+1}=w_*^{n+1}-\gradG\rho_*^{n+1}$, substituting this identity into
\eqref{continuity} yields
\begin{align*}
 \frac{\rho_*^{n+1}-\rho_*^n}{\tau}
 &=-\divG D_q\cH(\rho_*^{n+1},-\gradG\phi^n)
 +\DeltaG\rho_*^{n+1},
\end{align*}
which is \eqref{scheme-rho} for $\rho_*$.

Next, we take the variation with respect to $\rho^{k}$. 
Let $1\le k\le N-1$ and $\eta\in\bbR_0^d$.  
Stationarity with $\rho^k$ for $\mathfrak L_{\tau}$ in the tangent space gives
\begin{equation}\label{kkt-rho}
 \ip{\frac{\phi^k-\phi^{k-1}}{\tau}
 +D_\rho\cL(\rho_*^k,w_*^k)-D_\rho \cF(\rho_*^k)
 +\DeltaG\phi^{k-1}}{\eta}=0,
 \qquad \eta\in\bbR_0^d.
\end{equation}
By Lemma~\ref{legendre}, we  also have 
\begin{equation}\label{legendre-rho}
D_\rho\cL(\rho_*^k,w_*^k)=-D_\rho\cH(\rho_*^k,-\gradG\phi^{k-1}).
\end{equation}
Then  
\eqref{kkt-rho} becomes
\begin{equation}\label{value-stationarity}
 \ip{\frac{\phi^k-\phi^{k-1}}{\tau}
 -D_\rho\cH(\rho_*^k,-\gradG\phi^{k-1})
 -D_\rho \cF(\rho_*^k)+\DeltaG\phi^{k-1}}{\eta}=0,
 \qquad \eta\in\bbR_0^d.
\end{equation}
For $1\le k\le N-1$, define 
\begin{equation}\label{value-constant}
\begin{aligned}
 R^k&:=
 \frac{\phi^k-\phi^{k-1}}{\tau}
 -D_\rho\cH(\rho_*^k,-\gradG\phi^{k-1})
 -D_\rho \cF(\rho_*^k)+\DeltaG\phi^{k-1},\;
 c_k:=\frac1d\ip{R^k}{\bm1}.
\end{aligned}
\end{equation} 
For $1\le k\le N-1$, equation~\eqref{value-stationarity} implies $R^k\in(\bbR_0^d)^\perp=\operatorname{span}\{\bm1\}$. Hence $R^k=c_k\bm1$, and the constants $c_k$ are uniquely determined by the zero-mean multipliers $\{\phi^k\}_{k=0}^{N-1}$.  
At the terminal layer, a variation with $\rho_*^N$ gives
\begin{align}
 0=\ip{D_\rho \cU_T(\rho_*^N)-\phi^{N-1}}{\eta}
 +\tau\ip{D_\rho\cL(\rho_*^N,w_*^N)-D_\rho \cF(\rho_*^N)
 +\DeltaG\phi^{N-1}}{\eta}.
 \label{terminal-variation}
\end{align}
Define $\phi^N=D_\rho \cU_T(\rho_*^N)$.
Then \eqref{terminal-variation} and \eqref{legendre-rho} yield the same residual formula as in \eqref{value-constant} at $k=N$. Defining $R^N$ by that formula and $c_N=d^{-1}\ip{R^N}{\bm1}$ gives $R^N=c_N\bm1$. 

\textit{Step 3: We determine the constants $a_n$ in  \textit{Step 1}.}   
Set $a_N:=0$ and 
$\widehat\phi^N=\phi^N=D_\rho \cU_T(\rho_*^N)$.  Define the remaining constants by
$a_{k-1}=a_k+\tau c_k$ for $k=N,N-1,\ldots,1$.
Equivalently, $a_n=\tau\sum_{k=n+1}^N c_k$ for $0\le n<N$.
Thus every $a_n$ is uniquely determined by the KKT multiplier $\{\phi^k\}_{k=0}^{N-1}$
and the terminal condition $\phi^N$.
Define $\widehat\phi^k=\phi^k+a_k\bm1$ for $0\le k\le N$.  Since
$\gradG\bm1=0$ and $\DeltaG\bm1=0$, we have
$\gradG\widehat\phi^k=\gradG\phi^k$ and
$\DeltaG\widehat\phi^k=\DeltaG\phi^k$.  Therefore 
\begin{align*}
 &\frac{\widehat\phi^k-\widehat\phi^{k-1}}{\tau}
 -D_\rho\cH(\rho_*^k,-\gradG\widehat\phi^{k-1})
 -D_\rho\cF(\rho_*^k)+\DeltaG\widehat\phi^{k-1}\\
 &=
 \frac{\phi^k-\phi^{k-1}}{\tau}
 -D_\rho\cH(\rho_*^k,-\gradG\phi^{k-1})
 -D_\rho\cF(\rho_*^k)+\DeltaG\phi^{k-1}
 +\frac{a_k-a_{k-1}}{\tau}\bm1\\
 &=R^k+\frac{a_k-a_{k-1}}{\tau}\bm1
 =\bigl(c_k+\frac{a_k-a_{k-1}}{\tau}\bigr)\bm1=0.
\end{align*}
  Hence $(\widehat\phi,\rho_*)$ satisfies
\eqref{scheme-phi}.  Moreover, \eqref{continuity} and the flux relation
\[
 m_*^{n+1}
 =D_q\cH(\rho_*^{n+1},-\gradG\widehat\phi^n)
 -\gradG\rho_*^{n+1}
\]
give
\[
 \frac{\rho_*^{n+1}-\rho_*^n}{\tau}
 =-\divG D_q\cH(\rho_*^{n+1},-\gradG\widehat\phi^n)
 +\DeltaG\rho_*^{n+1},
\]
which is \eqref{scheme-rho}.  
Thus $(\widehat\phi,\rho_*)$ satisfies the system
\eqref{potential-scheme}.   Since the KKT multiplier is unique by \textit{Step 1}, the quantities $R^k$, $c_k$, $a_k$, and the corrected sequence $\widehat\phi$ are unique as well. 

\textit{Proof of (iii).} 
Conversely, let $(\phi,\rho)$ solve \eqref{potential-scheme} in the open
simplex and define $(w,m)$ by \eqref{optimal-flux}.
Here $\phi=(\phi^0,\ldots,\phi^N)$,
$\rho=(\rho^0,\ldots,\rho^N)$, and
$w=(w^1,\ldots,w^N)$, $m=(m^1,\ldots,m^N)$ are discrete paths. 
Then $\frac{\rho^{n+1}-\rho^n}{\tau}+\divG m^{n+1}=0$, so $(\rho,m)$ is feasible.
The identities in Lemma~\ref{legendre} give the analogue of
\eqref{kkt-momentum}. The value and terminal equations, read backwards
through the calculation above, give stationarity with respect to every
density variable.  Hence
$(\rho,m,\phi)$ satisfies the KKT conditions.

Let $(\bar\rho,\bar m)$ be any feasible pair.  The
action is convex in $(\rho,m)$ because $w=m+\gradG\rho$ is affine, $\cL$ is
jointly convex, $-\cF$ is convex, and $\cU_T$ is convex.  Therefore
\begin{align*}
 \mathcal A_\tau(\bar\rho,\bar m)-\mathcal A_\tau(\rho,m)
 &\ge D\mathcal A_\tau(\rho,m)[\bar\rho-\rho,\bar m-m]\\
 &=\tau\sum_{n=0}^{N-1}\ip{\phi^n}{\Gamma^n(\bar\rho,\bar m)-\Gamma^n(\rho,m)}=0.
\end{align*}
Here the second equality is primal stationarity and the last one uses
feasibility of both pairs.  Thus $(\rho,m)$ is a global minimizer, and it is
 the unique minimizer by Lemma~\ref{minimizer-lemma}.
\end{proof}

\section{Numerical realization}\label{sec6}
In this section, we present a numerical realization for solving the
discrete variational problem \eqref{min-prob}. We first introduce
mass-preserving coordinates, which transform the problem into a
finite-dimensional equality constrained optimization problem. We then
develop a feasible primal--dual Newton method in which every accepted
Newton step preserves the continuity equation and strict positivity of
the density.

To represent each skew-symmetric edge field by its independent
components, we fix an orientation of the graph.
Fix the oriented edge set
$E^+=\{(i,j)\in E:i<j\}$ and let $e=|E^+|=|E|/2$.
We identify an edge field $q\in\bbS^{d\times d}$ with
$(q_{ij})_{(i,j)\in E^+}\in\bbR^e$.  Let
$\mathsf G\in\bbR^{e\times d}$ be the weighted incidence matrix, so that
$\mathsf Gv$ is the graph gradient on $E^+$ and
$\divG q=-\mathsf G^Tq$. 
Let $\Pi\in\bbR^{d\times(d-1)}$ have orthonormal columns spanning
$\bbR_0^d$, and put $\bar\rho=d^{-1}\bm1$.  Every mass one vector has the
unique representation
\[
 \rho^n=\bar\rho+\Pi y^n,
 \qquad y^n=\Pi^T(\rho^n-\bar\rho).
\]
The initial coordinate $y^0$ is fixed.  For $1\le n\le N$, let
$m^n\in\bbR^e$ and set
$w^n=m^n+\mathsf G\rho^n$.  The continuity equation becomes
\begin{equation}\label{reduced-continuity}
 c^n(y,m)
 :=\frac{y^n-y^{n-1}}{\tau}
 -\Pi^T\mathsf G^Tm^n=0,
 \qquad 1\le n\le N.
\end{equation}
Indeed, the full continuity residual has zero sum, and its projection onto
$\bbR_0^d$ vanishes if and only if the residual itself vanishes.

Set
$ 
z=(y^1,m^1,\ldots,y^N,m^N)
$ 
and write the equations in \eqref{reduced-continuity} as $Az=b$, where
\begin{align*}
 (Az)^1&=\tau^{-1}y^1-\Pi^T\mathsf G^Tm^1,\quad
 (Az)^n=\tau^{-1}(y^n-y^{n-1})-\Pi^T\mathsf G^Tm^n,
 \qquad 2\le n\le N,\\
 b&=(\tau^{-1}y^0,0,\ldots,0)^T.
\end{align*}
The matrix $A$ has full row rank.  In fact, for arbitrary
$v=(v^1,\ldots,v^N)$, take $\delta m^n=0$, $\delta y^0=0$, and define
$\delta y^n=\delta y^{n-1}+\tau v^n$.  Then
$A\delta z=v$. 
In these coordinates, the discrete action \(\mathcal A_\tau\) in
\eqref{min-prob} is
\[
 J_\tau(z)
 =\tau\sum_{n=1}^N
 \bigl[\cL(\rho^n,w^n)-\cF(\rho^n)\bigr]
 +\cU_T(\rho^N).
\]
Thus \eqref{min-prob} is equivalent to
\begin{equation}\label{reduced-problem}
 \min\{J_\tau(z):Az=b,\ \rho_i^n(z)>0,\ 1\le n\le N,\ 1\le i\le d\}.
\end{equation}
Theorem~\ref{potential-sol} shows that this problem has a unique minimizer
$z_*$.  Its density is strictly positive, and  the path constructed in
Lemma~\ref{heat-competitor} supplies an interior point $z_0$ satisfying
$Az_0=b$.

For $\lambda=(\lambda^1,\ldots,\lambda^N)
\in(\bbR^{d-1})^N$, define
\[
 \mathscr L_\tau(z,\lambda)
 =J_\tau(z)+\tau\lambda^T(Az-b).
\]
The KKT equations for \eqref{reduced-problem} are
\begin{equation}\label{reduced-kkt}
 \nabla J_\tau(z)+\tau A^T\lambda=0,
 \qquad Az=b.
\end{equation}
Equivalently, \eqref{reduced-kkt} is the KKT system \eqref{kkt-system} written in the variables \((z,\lambda)\).  Below, we solve the constrained minimization problem \eqref{reduced-problem} by applying a feasible primal–dual Newton method to its KKT system \eqref{reduced-kkt}.  Once a KKT pair \((z,\lambda)\) is obtained, 
the value variable is recovered from the equality multipliers by
\(\phi^{n-1}=-\Pi\lambda^n+a_{n-1}\bm1\), where the constants are fixed backward as in the proof of Theorem~\ref{potential-sol}.

\subsection{A feasible Newton method for the KKT system}

We now describe a feasible primal--dual Newton method for solving \eqref{reduced-kkt}.  Assume in this subsection that $\cL$ is of class
$\mathcal C^2$ in the open simplex and that
$D_{ww}^2\cL(\rho,w)$ is positive definite.  These properties hold for
$\cL_\theta$, since the logarithmic mean is smooth and positive on
$(0,\infty)^2$.  Put
$ 
 H(z):=D^2J_\tau(z)
$
for $z=(y^1,m^1,\ldots,y^N,m^N)$. 
The matrix $H(z)$ is positive definite at every interior point.  To verify
this, let $\delta z$ be a nonzero variation and set
$\delta\rho^n=\Pi\delta y^n$ and
$\delta w^n=\delta m^n+\mathsf G\delta\rho^n$.  Then
\begin{align*}
 \delta z^TH(z)\delta z
 ={}&\tau\sum_{n=1}^N
 D^2\cL(\rho^n,w^n)
 [(\delta\rho^n,\delta w^n)]^2\\
 &-\tau\sum_{n=1}^N
 \ip{D^2\cF(\rho^n)\delta\rho^n}{\delta\rho^n}
 +\ip{D^2\cU_T(\rho^N)\delta\rho^N}{\delta\rho^N}.
\end{align*}
Since $\delta\rho^n=\Pi\delta y^n\in\bbR_0^d$, the strict convexity of
$-\cF$ is invoked only on the tangent space of the probability simplex,
exactly as assumed in Assumption~\ref{pot-struct}. 
Joint convexity of $\cL$ and convexity of $\cU_T$ make the first and last
terms on the right-hand side nonnegative.  If $\delta\rho^n\ne0$ at some level, the strict
concavity assumption on $\cF$ makes the second line positive.  If
$\delta\rho^n=0$ at every level, then $\delta z\ne0$ implies
$\delta m^n=\delta w^n\ne0$ at some level, and the positive definiteness of
$D_{ww}^2\cL$ again gives a positive value.

Let $(z,\lambda)$ be a current interior iterate satisfying $Az=b$.
Linearizing \eqref{reduced-kkt} gives
\begin{equation}\label{direct-newton}
 \mathcal K(z)
 \begin{pmatrix}\delta z\\ \delta\lambda\end{pmatrix}
 :=
 \begin{pmatrix}
  H(z)&\tau A^T\\
  \tau A&0
 \end{pmatrix}
 \begin{pmatrix}\delta z\\ \delta\lambda\end{pmatrix}
 =
 -
 \begin{pmatrix}
  \nabla J_\tau(z)+\tau A^T\lambda\\
  0
 \end{pmatrix}.
\end{equation}
This system has a unique solution.  Indeed, suppose that $ \mathcal K(z)(\delta z,\delta\lambda)^{\top}=0$. The second block equation gives
$A\delta z=0$.  Taking the inner product of the first block equation with
$\delta z$ yields $\delta z^TH(z)\delta z=0$, and hence $\delta z=0$.
The first block equation then gives $A^T\delta\lambda=0$.  Since $A$ has full
row rank, $A^T$ is injective and $\delta\lambda=0$.

The equality $A\delta z=0$ also shows that every candidate point 
$z+s\delta z$ along the Newton direction remains feasible:
$
 A(z+s\delta z)=b.
$ 
Let $\delta\rho^n=\Pi\delta y^n$.  To preserve strict positivity, choose
\begin{equation}\label{positive-step}
 s_{\max}
 =\min\biggl\{1,\,
 \eta\min_{\substack{1\le n\le N,\ 1\le i\le d\\
 \delta\rho_i^n<0}}
 \frac{-\rho_i^n}{\delta\rho_i^n}\biggr\},
 \qquad 0<\eta<1,
\end{equation}
where $s_{\max}=1$ if every component of $\delta\rho$ is nonnegative.
For $0<s\le s_{\max}$ and $\delta\rho_i^n<0$,
\[
 \rho_i^n+s\delta\rho_i^n
 \ge(1-\eta)\rho_i^n>0;
\]
the same conclusion is immediate when $\delta\rho_i^n\ge0$. 
Multiplying the first block equation of \eqref{direct-newton} by
$\delta z^T$ and using $A\delta z=0$ gives
\[
 \nabla J_\tau(z)^T\delta z
 =-\delta z^TH(z)\delta z.
\]
Hence every nonzero primal Newton direction is a descent direction for the
original action.  Starting from $s_{\max}$, we successively halve the step
until
\begin{equation}\label{direct-armijo}
 J_\tau(z+s\delta z)
 \le J_\tau(z)+\gamma s\nabla J_\tau(z)^T\delta z,
 \qquad 0<\gamma<\frac12.
\end{equation}
Differentiability of $J_\tau$ and the strict descent inequality show that
such a positive step exists.  Then the update is 
\[
 (z,\lambda)\leftarrow
 (z+s\delta z,\lambda+s\delta\lambda).
\]
The KKT residual $R_{\rm KKT}
:=
\max\big\{
\|\nabla J_{\tau}(z)+\tau A^T\lambda\|_{\ell^\infty},
\tau\|Az-b\|_{\ell^\infty}
\big\}$
is used as the stopping criterion.

Starting from the interior feasible point $z_0$ and any initial multiplier, every accepted
iterate satisfies the continuity equation, has unit mass, and remains in the
open simplex. 
With the variables grouped by time layer, \(H(z)\) is block diagonal and \(A\) is block bidiagonal. Hence the KKT matrix in \eqref{direct-newton} is block sparse in time. In the experiments, we solve this system by sparse direct factorization \cite{BenziGolubLiesen2005}.

\subsection{Numerical experiments}\label{experiment}
In this section, we consider the logarithmic-mean quadratic model given in
Example~\ref{exam}.  We take
$\cF(\rho)=-\alpha\|\rho\|^2/2$ and
$\cU_T(\rho)=\beta\|\rho-\rho_{\rm tar}\|^2/2$. 
All numerical experiments solve the KKT system
\eqref{reduced-kkt} by the feasible primal--dual Newton method
\eqref{direct-newton}--\eqref{direct-armijo}. In every test, the initial primal variable is obtained from the path in Lemma~\ref{heat-competitor}, and the initial
multiplier is set to zero. We use the positivity factor
$\eta=0.995$ and the Armijo parameter $\gamma=10^{-4}$.
The stopping tolerance is chosen between
$5\times10^{-11}$ and $2\times10^{-8}$ according to the size of the
problem. 

Let $(\widehat z,\widehat\lambda)$ be the computed KKT pair, and let
$(\widehat\rho,\widehat\phi)$ be the density and value variables recovered
from it.  We evaluate the following residuals of
\eqref{potential-scheme}: 
\begin{align*}
 R_\rho^n
 &=\frac{\widehat\rho^{n+1}-\widehat\rho^n}{\tau}
 -\divG\bigl(\theta(\widehat\rho^{n+1})
 \gradG\widehat\phi^n\bigr)-\DeltaG\widehat\rho^{n+1},\\
 R_\phi^n
 &=\frac{\widehat\phi^{n+1}-\widehat\phi^n}{\tau}
 -D_\rho\cH_\theta
 (\widehat\rho^{n+1},-\gradG\widehat\phi^n)
 -D_\rho \cF(\widehat\rho^{n+1})+\DeltaG\widehat\phi^n. 
\end{align*} 
Set
$R_w^n=\widehat w^{n+1}+\theta(\widehat\rho^{n+1})
\gradG\widehat\phi^n$, and
$R_T=\widehat\phi^N-D_\rho\cU_T(\widehat\rho^N)$.
At an exact KKT point, all four residuals vanish. 

\textbf{Accuracy and convergence tests on a three-vertex graph.} We first take a
three-vertex graph $V=\{1,2,3\}$ with edge weights $\omega_{12}=1$, $\omega_{23}=1.2$, and $\omega_{13}=0.8$, and set
\begin{equation*}
 T=1,\quad \alpha=0.7,\quad \beta=5,\quad
 \rho^0=(0.65,0.25,0.10)^T,\quad
 \rho_{\rm tar}=(0.15,0.25,0.60)^T.
\end{equation*}
We take $N=24$ time steps. Figure~\ref{triangle-figure} shows the computed
density and recovered value variables. The Newton iteration terminates after four steps with $R_{\rm KKT}=3.33\times10^{-16}$.    All density components
remain positive, with $\min_{n,i}\rho_i^n=0.1$, and mass is preserved: 
$ 
\max_n\left|\sum_i\rho_i^n-1\right|
=2.22\times10^{-16}.
$
Moreover, 
\[
\max\left\{
\max_n\|R_\rho^n\|_{\ell^\infty},
\max_n\|R_\phi^n\|_{\ell^\infty},
\max_n\|R_w^n\|_{\ell^\infty},
\|R_T\|_{\ell^\infty}
\right\}
=3.45\times10^{-15}.
\]

\begin{figure}[htbp]
 \centering
 \includegraphics[height=0.30\textheight,keepaspectratio]
 {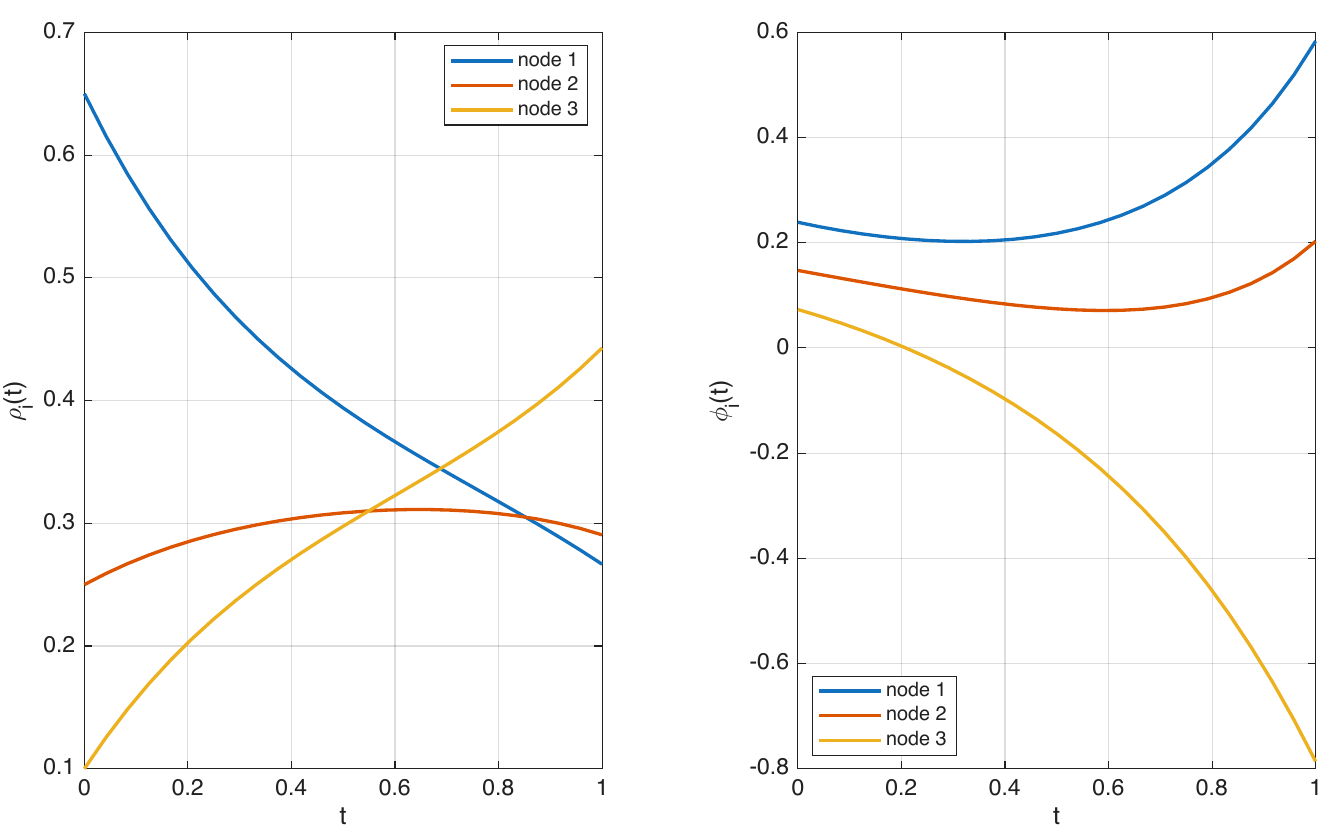}
 \caption{Density and value variables for the
 three-vertex graph experiment  
 with $N=24$.}
 \label{triangle-figure}
\end{figure}

For the temporal convergence test, we use the numerical solution computed by the same scheme with 
$N_{\rm ref}=1536$ time steps 
as the reference solution.  
At common time nodes, set
$E_\rho=\max_n\|\rho^n-\rho_{\rm ref}(t_n)\|$,
$E_\phi=\max_n\|\phi^n-\phi_{\rm ref}(t_n)\|$ and 
$E_p=\max_n\|\mathsf G(\phi^n-\phi_{\rm ref}(t_n))\|$. Table~\ref{time-table} exhibits first-order empirical convergence for
the three reported quantities. The observed first-order convergence of
$E_\rho$ and $E_\phi$ agrees with the maximum estimates in
Theorem~\ref{general-results}. 
Moreover, since 
$ (
\tau\sum_{n=0}^{N-1}
\|\gradG(\phi^n-\phi_{\rm ref}(t_n))\|_E^2
)^{1/2}
\le \sqrt{T}\,E_p,$ the observed first-order convergence of $E_p$ is stronger than, and therefore consistent with, the graph-gradient estimate in Theorem~\ref{general-results}. 

\begin{table}[htbp]
\centering
\caption{Time convergence for the triangle problem.}
\label{time-table}
\begin{tabular}{ccccccc}
\toprule
$N$ & $E_\rho$ & order & $E_\phi$ & order & $E_p$ & order\\
\midrule
$12$ & $1.4177\times10^{-2}$ & --      & $7.4725\times10^{-2}$ & --      & $1.2498\times10^{-1}$ & --\\
$24$ & $7.1265\times10^{-3}$ & $0.992$ & $3.7780\times10^{-2}$ & $0.984$ & $6.3291\times10^{-2}$ & $0.982$\\
$48$ & $3.5300\times10^{-3}$ & $1.014$ & $1.8745\times10^{-2}$ & $1.011$ & $3.1412\times10^{-2}$ & $1.011$\\
$96$ & $1.7137\times10^{-3}$ & $1.043$ & $9.1112\times10^{-3}$ & $1.041$ & $1.5269\times10^{-2}$ & $1.041$\\
\bottomrule
\end{tabular}
\end{table}

\textbf{Effects of graph topology on transport.} We next illustrate three effects of graph topology in Figures~\ref{cycle-figure}--\ref{tree-figure}: symmetric splitting between equivalent routes, a detour
created by deleted edges, and branching from one vertex to several leaves.
We first take the unit-weight cycle graph on $20$ vertices  with
\[
 V=\{1,\ldots,20\},\qquad
 E^+=\{(i,i+1):1\le i\le19\}\cup\{(1,20)\}.
\]
Vertices $1$ and $11$ are opposite, and the two paths joining them both have
ten edges.  Set $\vartheta_i=2\pi(i-1)/20$ and $\kappa=5$, and for $c\in[0,\pi],$ define
\[
 v_i(c)=\frac{\exp\bigl(\kappa\cos(\vartheta_i-c)\bigr)}
 {\sum_{j=1}^{20}\exp\bigl(\kappa\cos(\vartheta_j-c)\bigr)},
 \qquad
 \rho_i^0=0.98v_i(0)+\frac{0.02}{20},\qquad
 \rho_{{\rm tar},i}=0.98v_i(\pi)+\frac{0.02}{20}.
\]
Thus the initial and reference densities are centered at vertices $1$ and
$11$, respectively, and both  are strictly
positive.  We take $T=1$, $N=64$, $\alpha=1$, and
$\beta=5000$. 
The heat map in Figure~\ref{cycle-figure} shows two symmetric density bands
moving along the two symmetric paths and recombining near vertex $11$, illustrating the equilibrium splitting of the population between two
routes with identical costs. The
right panel compares $\rho^N$ with $\rho_{\rm tar}$, which nearly coincide.

\begin{figure}[htbp]
 \centering
 \includegraphics[width=0.9\textwidth]{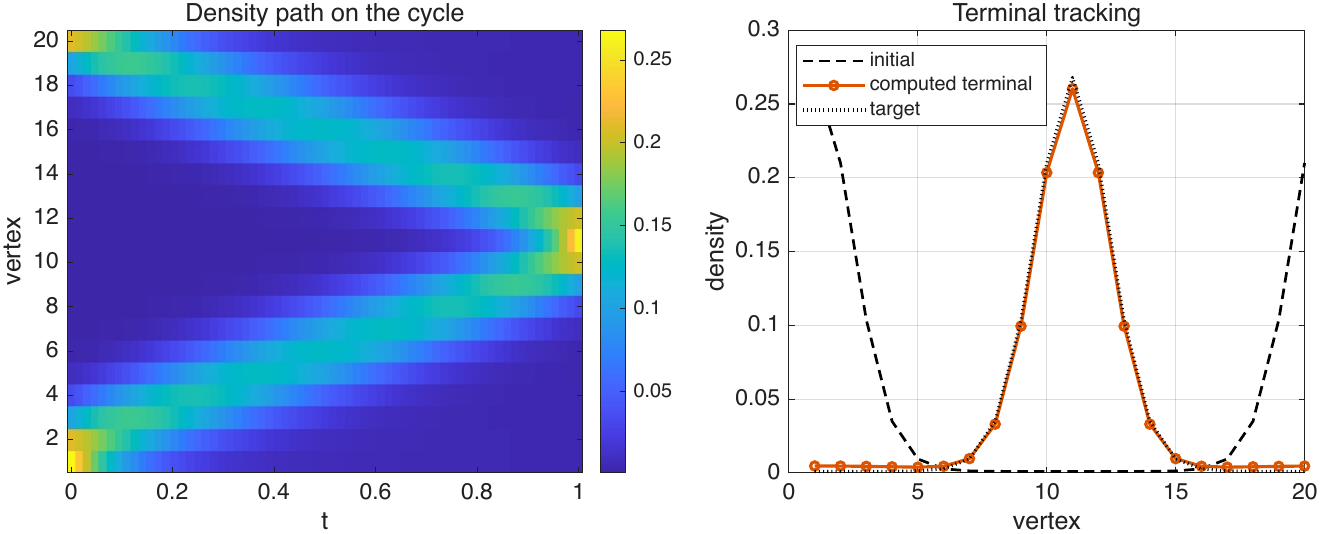}
 \caption{Transport between opposite vertices of a $20$-vertex cycle.  The left panel resolves the two symmetric routes around the graph; the right panel
 compares the computed terminal density with the penalized target. }
 \label{cycle-figure}
\end{figure}

Then we consider an obstacle test in Figure~\ref{perforated-figure} by removing five vertices from an $8\times8$ lattice to
form a vertical wall.  The connected graph has $59$ vertices and
$97$ edges.  The initial and target profiles lie on opposite sides near the
lower boundary.  Figure~\ref{perforated-figure} shows that the density must
move through the only passage, and the detour is caused solely by the edge set. This shows how removing edges
changes the admissible transitions and forces the equilibrium
population flow to pass through the remaining corridor.

\begin{figure}[htbp]
 \centering
 \includegraphics[width=0.86\textwidth]
 {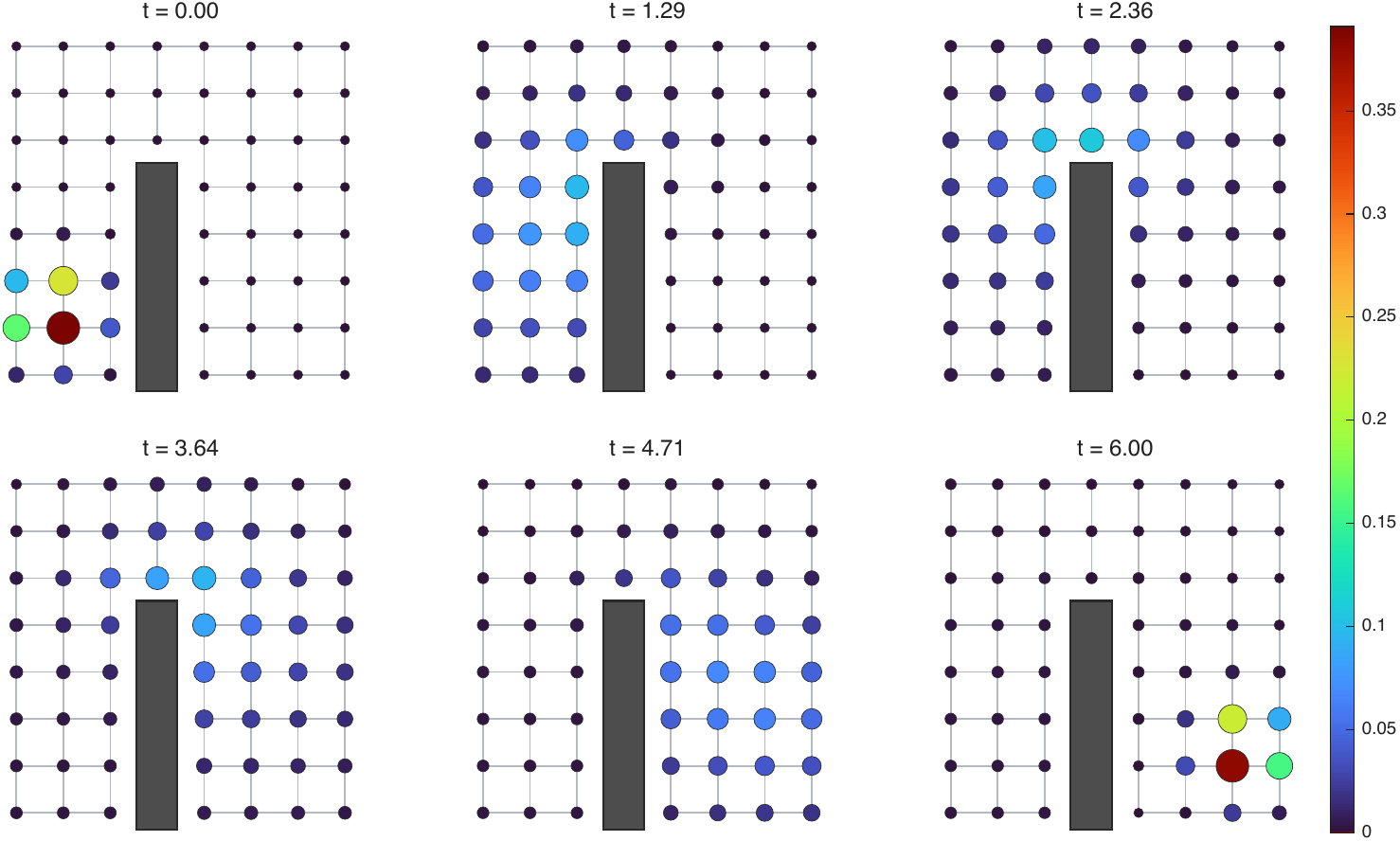}
 \caption{Transport on an  $8\times8$ lattice.  The gray wall marks
 the deleted vertices and edges, so the density is forced through the upper
 passage.}
 \label{perforated-figure}
\end{figure}

Next, consider a tree with six arms of length six in Figure~\ref{tree-figure}.  The initial
density is concentrated at the common center, while the terminal target gives
equal mass to the six leaves. Figure~\ref{tree-figure} resolves the symmetric propagation along the
arms and the formation of six terminal peaks. This reflects the
branching of the equilibrium population from the root to the six
leaves under a symmetric terminal cost.

\begin{figure}[htbp]
 \centering
 \includegraphics[width=0.86\textwidth]
 {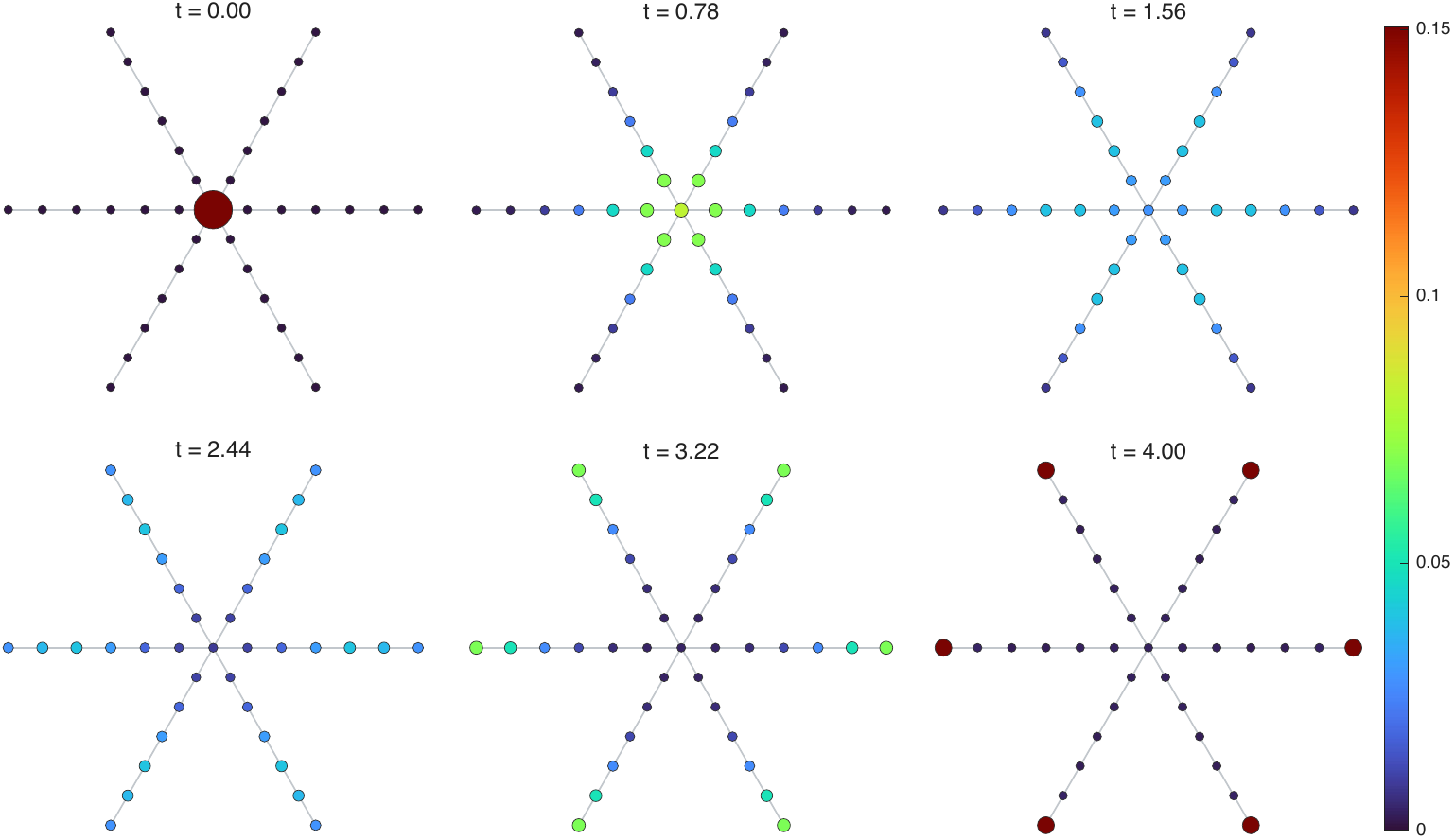}
 \caption{Bifurcation from one central peak to six leaf peaks on a $37$-vertex
 tree.  Marker area and color represent the density.}
 \label{tree-figure}
\end{figure}

\textbf{Congestion-dependent route choice.} 
Mean field games with congestion effects on graphs have been studied in \cite{Gueant2015}. We next apply the proposed method to a potential MFG on a graph with two alternative routes and examine how a congestion penalty changes the distribution of flow between them. 
As shown in Figure~\ref{route-figure}, we consider the $7\times7$ lattice with coordinates
$x,y\in\{(k-3)/3:k=0,\ldots,6\}$.  We delete the five vertices
$(0,k/3)$, $k=-2,\ldots,2$.  The remaining graph has
$44$ vertices and $68$ edges, with one passage above and one passage below
the wall.  For a center $c\in\bbR^2$, set
\[
 \varrho_i(c)=(1-\eta)
 \frac{\exp\big(-\norm{x_i-c}^2/(2\sigma^2)\big)}
 {\sum_{j=1}^{44}\exp\big(-\norm{x_j-c}^2/(2\sigma^2)\big)}
 +\frac{\eta}{44}, \quad i=1,2,\ldots,44.
\]
We use $\rho_i^0=\varrho_i(c_0)$ and
$\rho_{{\rm tar},i}=\varrho_i(c_T)$, where
$c_0=(-0.78,0.55)$, $c_T=(0.78,0.55)$, $\sigma=0.2$, and $\eta=0.04$.
The initial and target densities attain their maxima at \(\textbf{ S}=(-2/3,2/3)\) and \(\textbf{T}=(2/3,2/3)\), respectively, which are marked in Figure~\ref{route-figure}.  The shortest routes between these vertices
through the upper and lower passages contain, respectively, $6$ and $14$
edges.  We fix $T=6$, $N=24$, and $\beta=1000$,
 and vary the congestion
coefficient $\alpha$ in \eqref{quad-data}.  Here we recall that $\beta$ is the coefficient of the terminal penalty
$\cU_T(\rho)=\beta\norm{\rho-\rho_{\rm tar}}^2/2$. Since the running action contains
$\alpha\norm{\rho}^2/2$, larger $\alpha$ gives a stronger penalty for
concentrated density.

\begin{figure}[htbp]
 \centering
 \includegraphics[width=0.86\textwidth]
 {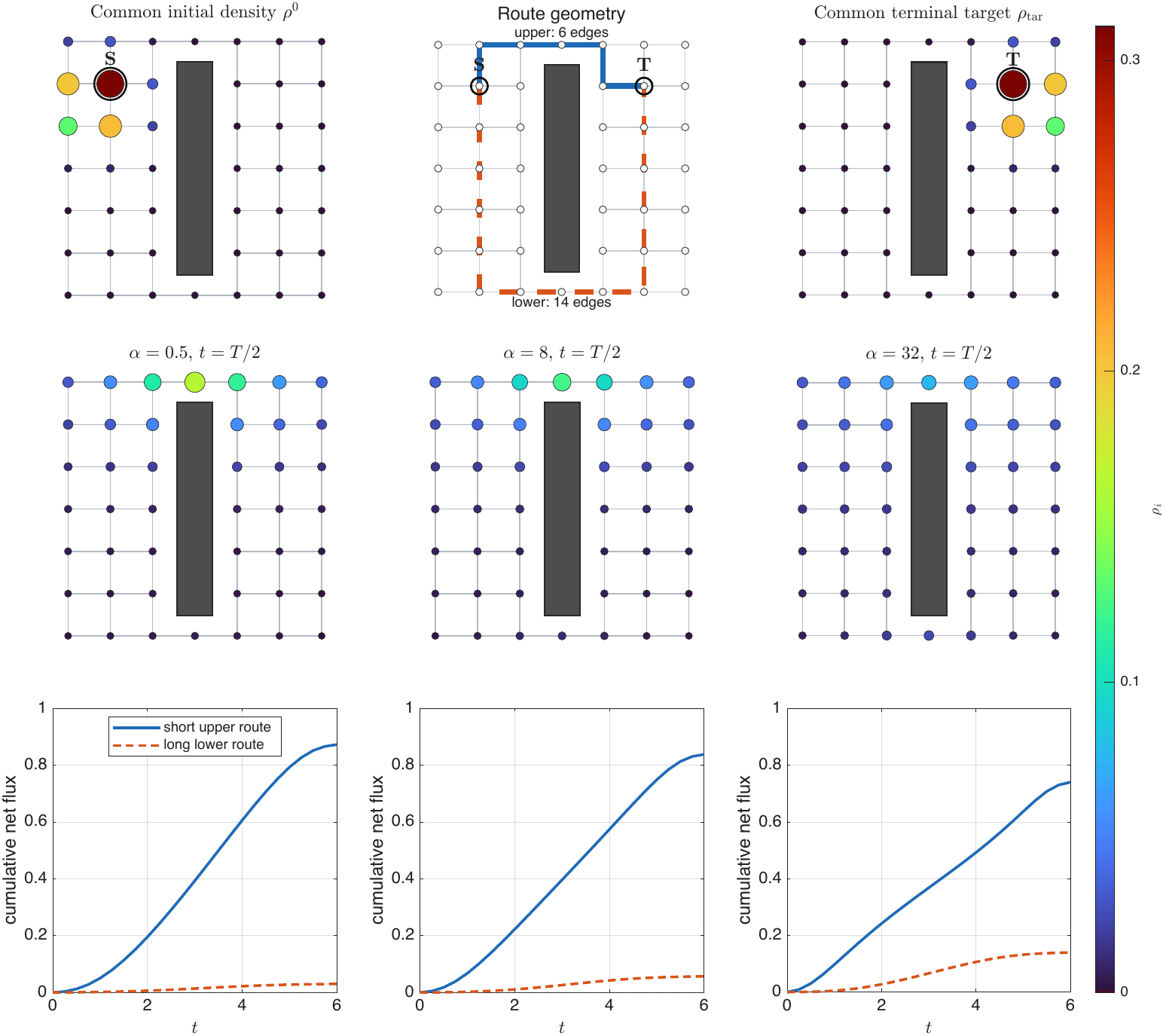}
 \caption{Congestion-dependent route choice on the two-passage graph.
 The upper row shows the common initial density, the two routes between their
 peak vertices, and the common terminal target.  The highlighted upper and
 lower routes contain $6$ and $14$ edges.  The middle row shows the density at
 $t=T/2$ for $\alpha=0.5,8,32$.  All density panels use the same scale; marker
 area and color increase with the vertex density.  The lower row shows the
 cumulative net flux \eqref{route-flux} through the upper and lower passages.
 The gray wall represents the deleted vertices and edges.}
 \label{route-figure}
\end{figure}

Let $e_+$ and $e_-$ denote, respectively, the left-to-right crossing edges in
the upper and lower passages.  With the orientation used in the computation,
their cumulative net fluxes are
\begin{equation}\label{route-flux}
 Q_\pm^n=-\tau\sum_{k=1}^n m_{e_\pm}^k, 
\end{equation} where \(m^k\in\mathbb R^{|E^+|}\) is the discrete edge-flux vector in the continuity equation, and \(m_{e_\pm}^k\) denotes its component corresponding to the crossing edge \(e_\pm\). 
For $\alpha=0.5,8,32$, the final upper fluxes are $0.8726$, $0.8381$, and
$0.7409$, whereas the lower fluxes are $0.0303$, $0.0569$, and $0.1398$.
Thus the lower-route share of the cumulative net crossing flux,
$Q_-^N/(Q_+^N+Q_-^N),$ 
equals $3.4\%$, $6.4\%$, and $15.9\%$, respectively. As $\alpha$ increases, a larger share of the computed net crossing flux uses the
longer lower route; see Figure~\ref{route-figure}. This behavior is
consistent with the stronger congestion penalty on concentrated density
induced by the density-dependent coupling.

\section{Conclusion}

This paper establishes convergence of a time-staggered scheme for MFGs with individual noise on finite graphs.
For general monotone coefficients,
we prove uniqueness, timestep-uniform interior estimates, and
first-order convergence for every interior discrete solution. For
potential MFGs, we further establish a variational characterization of
the scheme through a convex discrete action, which yields existence of
an interior discrete solution and supports a feasible primal--dual
Newton realization in mass-preserving coordinates. Several questions remain open. For example, 

\begin{itemize}
\item[(i)] 
The existence theory developed in this paper relies on the variational
structure of potential MFGs. Establishing existence for general monotone
graph MFGs remains open.

\item[(ii)] Developing higher-order
time discretizations that preserve the discrete fundamental identity and
the Lasry--Lions monotonicity mechanism, and the variational
structure in the potential case, would be of considerable interest.

\item[(iii)]
The present work treats the graph as the state space of the MFG.
It would also be interesting to understand how families of graph MFGs
approximate continuum or metric-graph models.
\end{itemize}

\section{Appendix}
\begin{proof}[Proof of Proposition~\ref{cont-sol}]
Assumption~\ref{continuous-data} and
\cite[Theorem~1.2]{GangboMunozWuZhang2026} give the existence of the solution.
If Assumption~\ref{assume-monotone} holds, then
\eqref{terminal-mono} holds and \eqref{differential-LL} implies
\eqref{LL}. The uniqueness assertion follows from the same theorem. 

According to 
\cite[Proposition~3.3 and Lemma~3.4]{GangboMunozWuZhang2026}, there are constants
$C_\phi,\delta_*>0$, depending on
$G,T$ and $\epsilon$,   such that
\[
 \max_{0\le t\le T}\norm{\phi(t)}_{\ell^\infty}\le C_\phi,
 \qquad
 \min_{0\le t\le T}\min_i\rho_i(t)\ge\delta_*.
\]
We may assume
$\delta_*\in(0,1/d)$.
Then 
\[
 (\rho(t),\gradG\phi(t))\in
 K_*:=\{(\rho,p):\rho\in\cP_{\delta_*},
 \ \norm{p}_{\ell^\infty(E)}
 \le2\sqrt{\omega_{\max}}C_\phi\}
\]
for every $t\in[0,T]$.  $K_*$ is a compact subset of
$\cP^\circ(G)\times\bbS^{d\times d}$, and thus the maps
$\bm H,\bm B$ and their first derivatives are bounded on $K_*$.
The first two equations in \eqref{mfg} give
$\norm{\dot\phi}_{L^\infty(0,T)}
+\norm{\dot\rho}_{L^\infty(0,T)}\le C_{K_*}$.
With $p=\gradG\phi$, differentiating the first two equations in
\eqref{mfg} with respect to time gives
\begin{align*}
 \ddot\phi
 &=D_\rho\bm H(\rho,p)[\dot\rho]
   +D_p\bm H(\rho,p)[\gradG\dot\phi]-\DeltaG\dot\phi,\\
 \ddot\rho
 &=\divG\bigl(D_\rho\bm B(\rho,p)[\dot\rho]
   +D_p\bm B(\rho,p)[\gradG\dot\phi]\bigr)+\DeltaG\dot\rho.
\end{align*}
The bounded first derivatives on $K_*$ now give uniform bounds for
$\ddot\phi$ and $\ddot\rho$.  
This
proves \eqref{cont-reg}.
\end{proof}

\begin{proof}[Verification of Remark~\ref{potential-monotonicity}]
Fix $(\rho,p)$, put $q=-p$, and let
$(\eta,r)\in\bbR_0^d\times \bbS^{d\times d}$.  From
\eqref{pot-coeff} and the chain rule, we have
\begin{align*}
 D_\rho\bm H(\rho,p)[\eta]
 &=D_{\rho\rho}^2\cH(\rho,q)\eta+D^2\cF(\rho)\eta,\;\;
 D_p\bm H(\rho,p)[r]
 =-D_{\rho q}^2\cH(\rho,q)r,\\
 D_\rho\bm B(\rho,p)[\eta]
 &=-D_{q\rho}^2\cH(\rho,q)\eta,\;\;
 D_p\bm B(\rho,p)[r]
 =D_{qq}^2\cH(\rho,q)r.
\end{align*}
Substituting these identities into the definition of
$\mathcal M_{\rho,p}$ in \eqref{differential-LL} gives
\begin{align*}
 \mathcal M_{\rho,p}(\eta,r)
={}&\ip{D_{\rho\rho}^2\cH(\rho,q)\eta}{\eta}
 +\ip{D^2\cF(\rho)\eta}{\eta}
 -\ip{D_{qq}^2\cH(\rho,q)r}{r}_E\\
 &\quad-\ip{D_{\rho q}^2\cH(\rho,q)r}{\eta}
 +\ip{D_{q\rho}^2\cH(\rho,q)\eta}{r}_E.
\end{align*}
The two mixed terms in the last line cancel, namely, $ 
\ip{D_{\rho q}^2\cH(\rho,q)r}{\eta}
=
\ip{D_{q\rho}^2\cH(\rho,q)\eta}{r}_E,$ 
since $\cH$ is scalar-valued and of class $\mathcal C^2$, so that both
expressions are equal to
$ \frac{\partial^2}{\partial s\,\partial t}
\cH(\rho+s\eta,q+tr)|_{s=t=0}$. 
Hence, by Assumption~\ref{pot-struct},  
\[
 \mathcal M_{\rho,p}(\eta,r)
 =\ip{D_{\rho\rho}^2\cH(\rho,q)\eta}{\eta}
 +\ip{D^2\cF(\rho)\eta}{\eta}
 -\ip{D_{qq}^2\cH(\rho,q)r}{r}_E<0
\]
whenever $(\eta,r)\ne(0,0)$.  Thus \eqref{differential-LL} holds.

Let $\rho,\sigma\in\cP^\circ(G)$.  Since $\cU_T$ is differentiable and
convex,
$
 \cU_T(\rho)
 \ge\cU_T(\sigma)+\ip{D_\rho\cU_T(\sigma)}{\rho-\sigma}.
$
By $g=D_\rho\cU_T$, this is
$
 \ip{g(\sigma)}{\rho-\sigma}
 \le\cU_T(\rho)-\cU_T(\sigma).
$
Similarly, interchanging $\rho$ with $\sigma$ gives
$
 \cU_T(\sigma)
 \ge\cU_T(\rho)+\ip{g(\rho)}{\sigma-\rho},
$ which yields
$\cU_T(\rho)-\cU_T(\sigma)\le\ip{g(\rho)}{\rho-\sigma}$.
Combining these inequalities gives
$\ip{g(\sigma)}{\rho-\sigma}
\le\cU_T(\rho)-\cU_T(\sigma)
\le\ip{g(\rho)}{\rho-\sigma}$, and hence
$\ip{g(\rho)-g(\sigma)}{\rho-\sigma}\ge0$.  This proves
\eqref{terminal-mono}.  Together with \eqref{differential-LL}, this
verifies Assumption~\ref{assume-monotone}.
\end{proof}

\begin{proof}[Verification for Example~\ref{exam}]
We verify Assumptions~\ref{continuous-data} and \ref{pot-struct} for
\eqref{quad-H} and \eqref{quad-data}.
For $r,s>0$,
$\theta(r,s)=\int_0^1r^ts^{1-t}\,\mathrm{d}t$.  If
$f_t(r,s):=r^ts^{1-t}$, then
$D^2f_t(r,s)[(a_0,b_0),(a_0,b_0)]=-t(1-t)f_t(r,s)
\bigl(\frac {a_0}{r}-\frac {b_0}{s}\bigr)^2\le0$.
The integral representation also shows that $\theta$ is smooth on $(0,\infty)^2$.  Hence $\theta$ is jointly concave, and so is
$\rho\mapsto\cH_\theta(\rho,q)$ for each fixed $q$.  Moreover,
$\ip{D_{qq}^2\cH_\theta(\rho,q)r}{r}_E
=\ip{\theta(\rho)r}{r}_E\ge\delta\norm{r}_E^2$ for
$\rho\in\cP_\delta$, because
$\theta(\rho_i,\rho_j)\ge\min\{\rho_i,\rho_j\}\ge\delta$.
Finally, $D^2\cF=-\alpha I$ on $\bbR_0^d$.
Together with the assumed convexity and $\mathcal C^2$ regularity of
$\cU_T$, these facts verify Assumption~\ref{pot-struct}. 

 It remains to
verify Assumption~\ref{continuous-data}.  The coefficients are
$\mathcal C^1$ in the open simplex.  The integral representation of
$\theta$ gives $\partial_1\theta\ge0$, and hence
$\bm H_i(\rho,p)\ge-\alpha\rho_i\ge-\alpha$.  Since
$\cH_\theta$ is one-homogeneous in $\rho$ and $\bm B(\rho,p)=\theta(\rho)p$,
$\ip{D_\rho\cH_\theta(\rho,-p)}{\rho}=\cH_\theta(\rho,p)$ and
$\ip{\bm B(\rho,p)}{p}_E=2\cH_\theta(\rho,p)$.
Therefore
$\ip{\bm B(\rho,p)}{p}_E-\ip{\bm H(\rho,p)}{\rho}
=\cH_\theta(\rho,p)+\alpha\norm{\rho}^2\ge0$.
With
$u=\rho_j/\rho_i$,
\[
 |\bm B_{ij}(\rho,p)|
 \le(\rho_i+\rho_j)\mathfrak a(u,p),\qquad
 \mathfrak a(u,p)=\norm{p}_{\ell^\infty(E)}
 \frac{|u-1|}{(1+u)|\log u|},
 \qquad
 \mathfrak a(1,p)=\frac12\norm{p}_{\ell^\infty(E)}.
\] 
The function $\mathfrak a$ is locally bounded and tends to zero as $u\to0^+$ or
$u\to\infty$, locally uniformly in $p$.  Finally,
$g=D_\rho\cU_T\in\mathcal C^1(\cP(G);\bbR^d)$ is bounded. 
Thus Assumption~\ref{continuous-data} holds with
$C_0=0$, $C_H=\alpha$, and
$
C_g=\sup_{\rho\in\cP(G)}
\norm{D_\rho\cU_T(\rho)}_{\ell^\infty}.$
Proposition~\ref{cont-sol} now gives the asserted unique solution and its estimates.
\end{proof}

\begin{proof}[Verification of Assumption~\ref{variational-data}
for $\cL_\theta$]
For $a>0$, the function $f(a,x)=x^2/a$ is jointly convex and
nonincreasing in $a$.  Since
$(\rho_i,\rho_j)\mapsto\theta(\rho_i,\rho_j)$ is concave, for
$0\le\lambda\le1$ one has
\begin{align*}
 &f\bigl(\theta(\lambda\rho_i+(1-\lambda)\sigma_i,
 \lambda\rho_j+(1-\lambda)\sigma_j),
 \lambda w_{ij}+(1-\lambda)z_{ij}\bigr)\notag\\
 &\quad\le f\bigl(\lambda\theta(\rho_i,\rho_j)
 +(1-\lambda)\theta(\sigma_i,\sigma_j),
 \lambda w_{ij}+(1-\lambda)z_{ij}\bigr)\notag\\
 &\quad\le\lambda f(\theta(\rho_i,\rho_j),w_{ij})
 +(1-\lambda)f(\theta(\sigma_i,\sigma_j),z_{ij}).
\end{align*}
The lower-semicontinuous convention in \eqref{quad-L} preserves this
inequality on the boundary and makes $\cL_\theta$ proper and lower
semicontinuous.  Summing over the edges proves joint convexity.  For each
fixed $\rho$, the resulting quadratic form is strictly convex in $w$ on its
effective domain.

Moreover, $0\le\theta(\rho_i,\rho_j)\le1$ on the simplex, so
$\cL_\theta(\rho,w)\ge\frac14\sum_{(i,j)\in E}w_{ij}^2
=\frac12\norm{w}_E^2$.
This proves both \eqref{coercivity} and \eqref{superlinear}, uniformly in
$\rho$.  Since $\cL_\theta(\rho,0)=0$ for every
$\rho\in\cP^\circ(G)$, \eqref{zero-action} holds.  If $\rho_i=0$, then
$\theta(\rho_i,\rho_j)=0$ for every $j\sim i$; finiteness of
$\cL_\theta(\rho,w)$ therefore forces $w_{ij}=0$.  Thus
\eqref{action-boundary} holds.  Finally, $\cL_\theta$ is smooth in the open
simplex.  Hence $\cL_\theta$ satisfies
Assumption~\ref{variational-data}.
\end{proof}

\begin{proof}[Proof of the quantitative estimate \eqref{strong-LL-ineq}]
Fix a compact set
$K\Subset\cP^\circ(G)\times\bbS^{d\times d}$.
Since the derivatives are linear in their directions,
$\mathcal M_{\xi,z}$ is homogeneous of degree two:
$\mathcal M_{\xi,z}(\lambda\eta,\lambda r)
=\lambda^2\mathcal M_{\xi,z}(\eta,r)$ for $\lambda\in\bbR$.
Since $K$ is compact and lies in the open simplex,
$\delta_K:=\min\{\xi_i:(\xi,z)\in K,\ i\in V\}>0$.
Let $\widehat K$ be the convex hull of $K$.  Every density component of
every point in $\widehat K$ is at least $\delta_K$.  Hence
$\widehat K$ is a compact subset of
$\cP^\circ(G)\times\bbS^{d\times d}$.  Set
$\mathcal S_1=\{(\eta,r)\in\bbR_0^d\times\bbS^{d\times d}:
\norm{\eta}^2+\norm{r}_E^2=1\}$.
The function $\mathcal M_{\xi,z}(\eta,r)$ is continuous and, by
\eqref{differential-LL}, strictly negative on the compact set
$\widehat K\times\mathcal S_1$.  Hence
$
 c_K:=-\max_{\substack{(\xi,z)\in\widehat K\\
                       (\eta,r)\in\mathcal S_1}}
 \mathcal M_{\xi,z}(\eta,r)>0.
$
For a nonzero $(\eta,r)$, put
$a=(\norm{\eta}^2+\norm{r}_E^2)^{1/2}$.  The homogeneity of
$\mathcal M_{\xi,z}$ then gives
\[
\begin{aligned}
 \mathcal M_{\xi,z}(\eta,r)
 &=a^2\mathcal M_{\xi,z}\bigl(\eta/a,r/a\bigr)
 \le-c_K\bigl(\norm{\eta}^2+\norm{r}_E^2\bigr),
 \qquad (\xi,z)\in\widehat K.
\end{aligned}
\]
The same inequality is immediate when $(\eta,r)=(0,0)$.

For $(\rho,p),(\sigma,q)\in K$, set
$\eta=\rho-\sigma$ and $r=p-q$.  Define the line segment
\[
 (\rho_s,p_s)=\bigl((1-s)\sigma+s\rho,(1-s)q+sp\bigr)
 =(\sigma+s\eta,q+sr),\qquad 0\le s\le1.
\]
Then $(\rho_s,p_s)\in\widehat K$, $(\rho_0,p_0)=(\sigma,q)$,
$(\rho_1,p_1)=(\rho,p)$, and
$\dot\rho_s=\eta$, $\dot p_s=r$.  Let
$\Psi(s)=\ip{\bm H(\rho_s,p_s)}{\eta}
-\ip{\bm B(\rho_s,p_s)}{r}_E$.
The chain rule yields
\[
\begin{aligned}
 \Psi'(s)
 ={}&\ip{D_\rho\bm H(\rho_s,p_s)[\eta]
          +D_p\bm H(\rho_s,p_s)[r]}{\eta}\\
 &-\ip{D_\rho\bm B(\rho_s,p_s)[\eta]
          +D_p\bm B(\rho_s,p_s)[r]}{r}_E
 =\mathcal M_{\rho_s,p_s}(\eta,r).
\end{aligned}
\]
Therefore, the fundamental theorem of calculus and the uniform bound above
give
\[
\begin{aligned}
 &\ip{\bm H(\rho,p)-\bm H(\sigma,q)}{\rho-\sigma}
 -\ip{\bm B(\rho,p)-\bm B(\sigma,q)}{p-q}_E\\
 &=\Psi(1)-\Psi(0)
 =\int_0^1\mathcal M_{\rho_s,p_s}(\eta,r)\,\mathrm{d}s\\
 &\le-c_K\bigl(\norm{\eta}^2+\norm{r}_E^2\bigr)
 =-c_K\bigl(\norm{\rho-\sigma}^2+\norm{p-q}_E^2\bigr).
\end{aligned}
\]
This proves \eqref{strong-LL-ineq}.
\end{proof}

\bibliographystyle{plain}
\bibliography{references_V4.bib}

\end{document}